\documentclass[reqno,a4paper,10pt]{amsart}
\DeclareMathAlphabet{\mathcal}{OMS}{cmsy}{m}{n}
\usepackage[latin1]{inputenc}
\usepackage[T1]{fontenc}
\usepackage[english]{babel}
\usepackage[backref,allcolors=blue,colorlinks=true, linkcolor=blue, bookmarks=true,plainpages]{hyperref}
\usepackage{caption}
\usepackage{subcaption}

\usepackage{indentfirst}
\usepackage{tablefootnote}
\usepackage[lmargin=3.8cm,tmargin=2.5cm,bmargin=2.5cm,rmargin=3.8cm]{geometry}

\usepackage{lipsum}

\usepackage[normalem]{ulem}

\usepackage[scaled=1.0076331,helvratio=0.9160502]{newtxtext}
\usepackage{times}
\usepackage{amsmath,amssymb,amsthm,amsfonts,mathtools}
\usepackage{mathptmx} 
\usepackage{dsfont} 
\usepackage{bbm}
\usepackage{pifont}

\usepackage{enumitem}

\usepackage{filecontents}

\newtheoremstyle{Def}
{5pt +1\p@ -2.0\p@}
{5pt +1\p@ -2.0\p@}
{\normalfont}			      
{}				  
{\bfseries}   
{.}               
{.4em}       
{}               
\theoremstyle{Def}
\newtheorem{definition}{Definition}[section]

\newtheoremstyle{theorem}
{5pt +1\p@ -2.0\p@}
{5pt +1\p@ -2.0\p@}
{\it}			      
{}				  
{\bfseries}   
{.}               
{.4em}       
{}  
\theoremstyle{theorem}
\newtheorem{theorem}{Theorem}[section]
\newtheorem{proposition}[theorem]{Proposition}

\newtheorem{lemma}[theorem]{Lemma}

\newtheorem{remark}[theorem]{Remark}

\numberwithin{equation}{section}

\newcommand{\R}{\mathbb R}

\newcommand{\N}{\mathbb N}
\newcommand{\Z}{\mathbb Z}

\makeatletter
\renewcommand{\tocsection}[3]{%
  \indentlabel{\@ifnotempty{#2}{\bfseries\ignorespaces#1 #2\quad}}\bfseries#3}
\renewcommand{\tocsubsection}[3]{%
  \indentlabel{\@ifnotempty{#2}{\ignorespaces#1 #2\quad}}#3}

\newcommand\@dotsep{4.5}
\def\@tocline#1#2#3#4#5#6#7{\relax
  \ifnum #1>\c@tocdepth 
  \else
    \par \addpenalty\@secpenalty\addvspace{#2}%
    \begingroup \hyphenpenalty\@M
    \@ifempty{#4}{%
      \@tempdima\csname r@tocindent\number#1\endcsname\relax
    }{%
      \@tempdima#4\relax
    }%
    \parindent\z@ \leftskip#3\relax \advance\leftskip\@tempdima\relax
    \rightskip\@pnumwidth plus1em \parfillskip-\@pnumwidth
    #5\leavevmode\hskip-\@tempdima{#6}\nobreak
    \leaders\hbox{$\m@th\mkern \@dotsep mu\hbox{.}\mkern \@dotsep mu$}\hfill
    \nobreak
    \hbox to\@pnumwidth{\@tocpagenum{\ifnum#1=1\bfseries\fi#7}}\par
    \nobreak
    \endgroup
  \fi}
\AtBeginDocument{%
\expandafter\renewcommand\csname r@tocindent0\endcsname{0pt}
}
\def\l@subsection{\@tocline{2}{0pt}{2.5pc}{5pc}{}}
\makeatother

\begin{document}

\title[Fourier transform decay on Hardy-Morrey spaces]{\large{Fourier transform decay of distributions in Hardy-Morrey spaces}}

\author {Marcelo F. de Almeida}
\address{Departamento de Matem\'atica, Universidade Federal de Sergipe, S\~ao Crist\'ov\~ao, SE,
49000-000, Brasil}
\email{marcelo@mat.ufs.br}

\author {Tiago Picon}
\address{Departamento de Computa\c{c}\~ao e Matem\'atica, Universidade S\~ao Paulo, Ribeir\~ao Preto, SP, 14040-901, Brasil}
\email{picon@ffclrp.usp.br}

\thanks{Work supported in part by CNPq (409306/2016-9) and FAPESP (2013/17636-5 and 2018/15484-7).}


\begin{abstract}
In this {paper} we {establish} decay estimates for Fourier transform on Hardy-Morrey spaces and its localizable version. Our work include some aspects to these spaces linked up with {pointwise Fourier estimates}, in particular  {a natural approach} on cancellation moment conditions.  As application, {we discuss the optimality}
for continuity of Fourier multipliers and pseudodifferential operators in Hardy-Morrey spaces.

\end{abstract}

\maketitle

\tableofcontents

\section {Introduction}
The theory of  Morrey spaces $\mathcal{M}^{\lambda}_{q}(\R^{n})$  have been extensively developed during the last decades in the several settings 
(see e.g. \cite{Adams,Adams0,Giga,Ta}). 
These spaces describe local regularity of functions in $L^q_{loc}(\R^n)$ by special averages on balls as a refinement of Lebesgue spaces. Among {many} applications of Morrey spaces  we highlight the classical gain of Sobolev embedding into H\"older function space $C^{0,\delta}(\R^{n})$ archived by Morrey (see \cite{Adams0} for details). In the other hand, it is well known that Hardy space $H^{q}(\R^n)$ (\cite{FS}) are nice substitutes of  Lesbegue spaces and that a rich functional distributional space shines for $0<q\leq 1$ with {several} applications in Harmonic Analysis, Functional Analysis and PDEs (see e.g. \cite{Garcia,S}).

In \cite{Tanaka, HH} {the authors} considered an hybrid space  $\mathcal{HM}^{\lambda}_{q}(\R^{n})$ for $0< q\leq \lambda <\infty$, called the Hardy-Morrey space, presenting a natural characterization via smooth maximal function and atomic decomposition theorem (see \cite[p.100]{HH} and \cite[Theorems 1.3 and 1.4]{Tanaka}). As expected, these spaces recover the Morrey spaces for  $1<q<\lambda<\infty$ and Hardy spaces for $0<q=\lambda<1$ {(see \cite{FS})}. Variations of Hardy-Morrey spaces have been studied in \cite{akb, K-p-Ho, K-p-Ho2} and explored in several {contexts} as Calder\'on-Zygmund theory, singular integral operators \cite{WJ} and characterizations of trace-law to Riesz potentials  \cite{Xiao}. These studies are motivated by applications on regularity theory of PDEs, geometric harmonic/potential analysis  and fluid dynamics  as well. {We point out that comparable to Morrey space $\mathcal{M}^{\lambda}_{q}(\R^{n})$ some properties for $\mathcal{HM}^{\lambda}_{q}(\R^{n})$} breaks down when {$0<q<\lambda<1$}. For instance, the Hardy-Morrey spaces are not closed by multiplication on test functions and for this reason some linear operators (as pseudodifferential operators) in general are not bounded on $\mathcal{HM}^{\lambda}_{q}(\R^{n})$.  It follows by intrinsic relation with Hardy spaces empowered with suitable nature of  Morrey norm.

In this paper we carry further to study of decay estimates for Fourier transform of distributions in Hardy-Morrey spaces (see Definition \ref{defHM}) and its localizable version $h\mathcal{M}_q^{\lambda}(\mathbb{R}^n)$ (see Definition \ref{defhM}). 
Our main result is the following:
\begin{theorem}
Let {$0<q \leq \lambda \leq 1$}.
If $f\in\mathcal{HM}_{q}^{\lambda}(\R^n)$  then there exists $C=C(n,q,\lambda)>0$ such that $\widehat{f}(\xi)$ is a continuous function and {satisfies} the {pointwise Fourier transform decay} 
\begin{equation}\label{decay-Fourier2a}
|\widehat{f}(\xi) | \leq C |\xi|^{n\left(\frac{1}{\lambda}-1 \right)} \|f \|_{\mathcal{HM}_{q}^{\lambda}}, {\quad  \forall  \,\,\xi \in \R^n}.
\end{equation}
If  $f\in h\mathcal{M}_{q}^{\lambda}(\R^n)$ then its  Fourier transform satisfies 
\begin{equation}
|\widehat{f}(\xi) | \lesssim  \langle\xi\rangle^{n\left(\frac{1}{\lambda}-1 \right)} \|f \|_{h\mathcal{M}_{q}^{\lambda}}, 
\end{equation}
where $\langle\xi\rangle:=(1+\vert\xi\vert^2)^{1/2}$.
\end{theorem}
The proof of the previous theorem is presented in Theorem  \ref{cor-decay} for distributions in Hardy-Morrey spaces and Theorem  \ref{cor-decay2}  for localizable version. As we can see some aspects of these spaces linked up to a suitable analysis on the necessity of moment conditions inspired by works due to Goldberg \cite{G} and Hounie \& Kapp \cite{HK}. {Also, as we can see by Remark \ref{rem1},  the decay estimates  for Fourier transform are not expected in Morrey spaces as $1\leq q<\lambda<\infty$}. Our approach follows by using the classical maximal function definition of the Hardy-Morrey spaces contributing to simplify several proofs found in the literature. In addition, {closely related to  Bownik \& Wang \cite{Bownik}} we present an analysis on moment conditions on Hardy-Morrey spaces  and cancellation properties on its localizable version. 

The organization of the paper is as follows. In the Section \ref{se2a}, we present general aspects of Hardy-Morrey spaces, in special we prove the necessity of moment conditions for functions in $\mathcal{HM}^{\lambda}_{q}(\R^{n})$ linked up with the Fourier transform {decay} for integrable fucntions given by Proposition \ref{prop2.5}.  {In Section \ref{Point-parte1} is devoted to the extension of Fourier transform decay (see Theorem \ref{cor-decay}) 
\begin{equation}\label{key-dec-int}
|\widehat{f}(\xi) | \leq C  |\xi|^{n\left(\frac{1}{\lambda}-1 \right)} \|f \|_{\mathcal{HM}_{q}^{\lambda}}, {\quad  \forall  \,\,\xi \in \R^n}
\end{equation}
for distributions $f$ in $\mathcal{HM}^{\lambda}_{q}(\R^{n})$ as $0<q \leq \lambda\leq 1$.  The proof is a careful analysis derived from Propositions \ref{prop2.5} and \ref{prop3.4} combined with 
atomic decomposition theorem \cite{HH}} (see Theorem \ref{atomich}). {Moreover, in Remark \ref{rem1} and Proposition \ref{exam.53} we show that the constraint  $0<q \leq \lambda<1$ 
is necessary in Theorem \ref{cor-decay}. 
As an application of the estimate \eqref{key-dec-int}, we discuss  an optimality of the parameters in the boundedness of certain Fourier multipliers on Hardy-Morrey spaces at Proposition \ref{optimal_Sigma2}.
Section \ref{sec4} is concerned to the localizable Hardy-Morrey spaces with emphasis to a natural extension of Fourier transform decay \eqref{key-dec-int} for distributions in $h\mathcal{M}_{q}^{\lambda}(\R^n)$, namely, 
\begin{equation} \label{key-dec-int2}
|\widehat{f}(\xi) | \leq C  \langle \xi\rangle^{n\left(\frac{1}{\lambda}-1 \right)} \|f \|_{h\mathcal{M}_{q}^{\lambda}}, {\quad  \forall  \,\,\xi \in \R^n}
\end{equation}
as $0<q\leq\lambda\leq 1$ (see Theorem \ref{cor-decay2}). This result strengthen the standard Fourier transform decay \cite[Proposition 5.1]{HK} on localizable Hardy space $h^{p}(\mathbb{R}^n)$ for $0<p\leq 1$. We also present some considerations on the optimality for continuity of standard pseudodifferential operators.

\section{Hardy-Morrey space \texorpdfstring{$\mathcal{HM}_q^\lambda$}{lh} }\label{se2a}
In this section we recall some properties and aspects of  Hardy-Morrey spaces and its nonhomogeneous version linked up to a suitable analysis on the necessity of moment conditions that can not be found in the literature.   A measurable function $f\in L^1_{loc}(\mathbb{R}^n)$ belongs to Morrey space $\mathcal{M}^\lambda_q(\mathbb{R}^{n})$ if it satisfies
\begin{equation}
\Vert f \Vert_{\mathcal{M}^\lambda_q}:=\sup_{Q}\vert Q\vert^{{1}/{\lambda}-{1}/{q}}\left( \int_Q\vert f(y)\vert^qdy \right)^{1/q}<\infty,  \nonumber
\end{equation} 
for $0<q\leq \lambda<\infty$, where the supremum is taken over all cubes $Q$ with sidelength $\ell_Q>0$ and Lebesgue measure $|Q|$.  The Morrey spaces endowed by functional $\Vert \cdot\Vert_{\mathcal{M}^{\lambda}_{q}}$ are complete metric spaces and include strictly  (see \cite{CTZ}) the Lebesgue spaces $L^q(\R^n)$. 
However these spaces differ in several points, for instance, 
the functions in $\mathcal{M}^\lambda_q(\mathbb{R}^{n})$ can not be approximated by smooth functions as long as $1\leq q<\lambda<\infty$. 
We refer \cite{Adams, Adams0,Giga,Ta} for an introduction on  these spaces for $1 \leq  q<\lambda<\infty$ and \cite[Chapter 3]{Triebel} for $0<q<\lambda<\infty$.

Jia and Wang had introduced the following class of tempered distributions.
\begin{definition}\label{defHM}Let $0<q\leq\lambda<\infty$. We say that a tempered distribution $f \in \mathcal{S}'(\mathbb{R}^{n})$ belongs to $\mathcal{HM}_q^{\lambda}(\mathbb{R}^{n})$, if there exists $\varphi \in \mathcal{S}(\mathbb{R}^{n})$ with  $\displaystyle{\int_{\R^{n}} \varphi(x)dx \neq 0}$ such that  the smooth maximal function
	$M_{\varphi}f\in \mathcal{M}_q^{\lambda}(\mathbb{R}^{n})$, where 
	\begin{equation}
	M_{\varphi}f(x)=\sup_{0<t<\infty} \left\vert \varphi_{t}\ast f(x)\right\vert, \quad \quad \varphi_t(x)=t^{-n}\varphi(x/t).
	\end{equation}
\end{definition}
The functional $\|f\|_{\mathcal{HM}_q^{\lambda}}:=\|M_{\varphi}f\|_{\mathcal{M}_{q}^{\lambda}}$ defines a quasi-norm as $0<q<1$ and is a norm for $q\geq 1$ (we always refer as a ``norm'' for simplicity).  It follows that $\mathcal{HM}_q^{\lambda}(\R^n)$ endowed by distance 
$d(f, g):=\|f-g\|_{\mathcal{HM}_q^{\lambda}}$ is a complete metric space \cite[Lemma 2.5]{HH}. We point out that $ \mathcal{HM}_{q}^{\lambda}(\mathbb{R}^n)$ coincides with $\mathcal{M}_{q}^{\lambda}(\mathbb{R}^n)$ as $1<q\leq \lambda<\infty$, i.e. there exist  constants $c_{1}, c_{2}>0$ depending only the parameters $n,q, \lambda$  such that   
\begin{equation}\label{hm}
c_{1}\|f\|_{\mathcal{HM}_q^{\lambda}}\leq \|f\|_{\mathcal{M}_q^{\lambda}} \leq c_{2} \|f\|_{\mathcal{HM}_q^{\lambda}}. 
\end{equation}
Moreover, the  continuous inclusion $\mathcal{HM}_{1}^\lambda(\R^n)\hookrightarrow \mathcal{M}_{1}^{\lambda}(\R^n)$ holds for $1<\lambda<\infty$ 
(see \cite{Tanaka,Xiao}). Obviously some properties are inherited from Morrey spaces, for instance the convexity in Morrey spaces are naturally extended by Hardy-Morrey spaces i.e. 
\begin{equation}
\Vert f\Vert_{\mathcal{HM}^{\lambda_2}_{q_2} }\leq \Vert f\Vert_{\mathcal{HM}^{\lambda_1}_{q_1}}^{\theta} \Vert f\Vert_{\mathcal{HM}^{\lambda_3}_{q_3}}^{1-\theta},\nonumber
\end{equation}
where $0<q_1<q_2<q_3 \leq \infty$ and  $0<\lambda_1<\lambda_2<\lambda_3 \leq  \infty$  satisfy 
$\frac{1}{q_2}=\frac{\theta}{q_1}+\frac{1-\theta}{q_3}$ and  ${\frac{1}{\lambda_{2}}=\frac{\theta}{\lambda_{1}}}+\frac{1-\theta}{\lambda_{3}}$ for $0<\theta<1$. {Now, consider $\delta_{R}f(x):= f(Rx)\,$ for $R>0$. Since $\varphi_{t} \ast \delta_{R}f={\delta_R (\varphi_{tR} \ast f)}$, it follows from scaling property of Morrey space $\mathcal{M}_q^\lambda(\R^n)$ that
	\begin{equation}\label{scaling}
		\|\delta_{R}f\|_{\mathcal{HM}_q^\lambda}=\|   M_{\varphi}(\delta_{R}f)\|_{\mathcal{M}_q^\lambda}=\| \delta_{R}(M_{\varphi}f)\|_{\mathcal{M}_q^\lambda}=R^{-n/\lambda}\|f\|_{\mathcal{HM}_q^\lambda}, 
	\end{equation}
	which is the expected scaling for  {distributions in}  Hardy-Morrey space $\mathcal{HM}_q^\lambda(\R^n)$.}

Analogous to Hardy spaces another maximal characterizations are obtained for Hardy-Morrey spaces, that will be describe in sequel. For each  $N \in \mathbb{N}_0:=\N \cup \left\{0\right\}$ let  $\mathcal{F}:=\{ \|\cdot\|_{\alpha,\beta}\,:\, \alpha,\beta\in \mathbb{N}_0^n\, \text{ such that }\,\vert \alpha\vert\leq N,\, \vert\beta\vert\leq N\, \}$ be a finite collection of semi-norms defined {on} $\mathcal{S}(\mathbb{R}^n)$ and consider
$\mathcal{S}_{\mathcal{F}}:=\left\{ \varphi \in \mathcal{S}(\mathbb{R}^n) \, :\, \|\varphi\|_{\alpha,\beta} =\sup_{x\in\mathbb{R}^n}\left\vert x^{\alpha}\partial_x^{\beta} \varphi(x)\right\vert \leq 1, \, \text{ for all }\, \|\cdot\|_{\alpha,\beta} \in \mathcal{F}   \right\}. $
We define the \textit{grand maximal function} of $f$ by 
$\displaystyle{M_{\mathcal{F}}f(x):=\sup_{\varphi \in \mathcal{S_{F}}} M_{\varphi}f(x)}$
and the \textit{non-tangential} version of $M_{\varphi}$  is defined by  
\begin{equation}\label{nont_max}
M_{\varphi}^{\ast}f(x):=\sup_{\vert x-y\vert<t} \left\vert \varphi_t\ast f(y)\right\vert.
\end{equation}
Clearly $M_{\varphi}f(x)\leq M^{\ast}_{\varphi}f(x)$ and $M_{\varphi}f(x)\lesssim \footnote{The notation $f \lesssim g $ means that there exists a constant $C>0$ such that $f(x)\leq C g(x)$ for all $x \in \R^n$} M _{\mathcal{F}}f(x)$ for every $x \in \R^{n}$. Lastly, we say that $f \in \mathcal{S}'(\mathbb{R}^{n})$ is a bounded tempered distribution if $f \ast \varphi \in L^{\infty}$ for all $\varphi \in \mathcal{S}(\R^{n})$. The following characterization can be stated for Hardy-Morrey spaces.
\begin{theorem}[\cite{HH}]\label{thm-eqHardy} 
Let $0< q\leq \lambda < \infty$ and $f \in \mathcal{S}'(\mathbb{R}^{n})$. The following statements are {equivalents}:
\begin{itemize}
	\item[(i)] There is $\varphi \in \mathcal{S}(\mathbb{R}^n)$ with $\int_{\R^n} \varphi(x) dx=1$ such that  $ M_{\varphi}f \in \mathcal{M}_{q}^{\lambda}(\mathbb{R}^n)$.
        \item[(ii)] There exist a collection $\mathcal{F}$ so that $M_{\mathcal{F}}f\in\mathcal{M}_{q}^{\lambda}(\mathbb{R}^n)$.
   \item[(iii)]  f is a bounded tempered distribution and $M^{\ast}_{\varphi}f \in\mathcal{M}_{q}^{\lambda}(\mathbb{R}^n)$.
	\end{itemize}
Moreover,
\begin{equation}\label{maximalop1}
\Vert M_{\mathcal{F}}f\Vert_{\mathcal{M}^{\lambda}_{q} } \lesssim  \Vert M^{\ast}_\varphi f\Vert_{\mathcal{M}^{\lambda}_{q}} \lesssim \Vert M_\varphi f\Vert_{\mathcal{M}^{\lambda}_{q}} \lesssim \Vert M_{\mathcal{F}}f\Vert_{\mathcal{M}^{\lambda}_{q}}.
\end{equation}	   
\end{theorem}
Hence, $f$ belongs to $\mathcal{HM}_{q}^{\lambda}(\mathbb{R}^n)$ if one (thus all) of the properties above are satisfied and  from Theorem \ref{thm-eqHardy}(ii)  we remark that  $\|M_{\varphi}f\|_{\mathcal{M}_{q}^{\lambda}}$ is independent of the choice of the function $\varphi$. As a consequence of Definition \ref{defHM}, the Hardy-Morrey spaces cover the Hardy spaces i.e $ \mathcal{HM}_{q}^{q}(\mathbb{R}^n) = H^{q}(\mathbb{R}^n) $ for $0<q<\infty$
and naturally these spaces are nice substitutes of Morrey spaces and describe a local control of Hardy spaces when $0<q\leq 1$ and $q <\lambda$. By reasons of previous comments, we emphasize our study on $\mathcal{HM}_q^{\lambda}(\R^n)$ with $0<q\leq 1$ and $q\leq \lambda$.

{Analogous to previous maximal functions, we may define the truncated version of} \textit{tangential, non-tangential} and \textit{grand maximal function} for $f \in \mathcal{S}'(\mathbb{R}^n)$ respectively as follows 
\begin{align*}
m_{\varphi}f(x)=\sup_{0<t\leq 1}\left\vert (\varphi_t\ast f)(x)\right\vert, \,\,
m_{\varphi}^{\ast}f(x)=\sup_{\vert x-y\vert<t} \left\vert (\varphi_t\ast f)(y)\right\vert \; \text{and} \;
m_{\mathcal{F}}f(x):=\sup_{\varphi \in \mathcal{S_{F}}} m_{\varphi}f(x),
\end{align*}
and the following nonhomogeneous version of $\mathcal{HM}_{q}^{\lambda}(\mathbb{R}^n)$.
\begin{definition}\label{defhM}
{Let $0<q\leq\lambda<\infty$.} We say that a tempered distribution $f \in \mathcal{S}'(\mathbb{R}^{n})$ belongs to $h\mathcal{M}_q^{\lambda}(\mathbb{R}^{n})$,  called nonhomogeneous (or localizable) Hardy-Morrey spaces, if there exists $\varphi \in \mathcal{S}(\mathbb{R}^{n})$ with  $\int_{\mathbb{R}^n}\varphi(x) dx \neq 0$ such that  
 $m_{\varphi}f\in \mathcal{M}_q^{\lambda}(\mathbb{R}^{n})$. 
 \end{definition}
The functional $\|f\|_{h\mathcal{M}_q^{\lambda}}:=\|m_{\varphi}f\|_{\mathcal{M}_{q}^{\lambda}}$ defines a quasi-norm as $0<q<1$ 
and the space is a complete metric space. 
Clearly the nonhomogeneous version covers the localizable Hardy spaces due to Goldberg \cite{G} i.e. $ h\mathcal{M}_{p}^{p}(\R^n)= h^{p}(\R^n) $,
the continuous inclusion $\mathcal{HM}_q^{\lambda}(\R^n) \hookrightarrow h\mathcal{M}_q^{\lambda}(\R^n)$ holds and  moreover $h\mathcal{M}_{q}^{\lambda}(\R^n)$ coincides with Morrey spaces  as $1<q <\lambda<\infty$. 
An analogous characterization for $h\mathcal{M}_q^{\lambda}(\mathbb{R}^{n})$, as Theorem \ref{thm-eqHardy}, can be stated with suitable adaptation for truncated maximal functions



The next assertion provides a way of transferring properties from $\mathcal{HM}_q^{\lambda}(\R^n)$ to the corresponding localizable space $h\mathcal{M}_q^{\lambda}(\R^n)$. This lemma was first announced in \cite[Proposition 4.8]{Sawano}, and here we present only some steps of the proof by convenience.

\begin{lemma}[\cite{Sawano}]\label{lem-glob_loc}
Let $\psi\in \mathcal{S}(\mathbb{R}^n)$, $\int_{\mathbb{R}^n} \psi(x) dx=1$ and $\int_{\mathbb{R}^n} x^{\alpha}\psi(x)dx=0$ for all multi-index $\alpha\in\mathbb{N}^n_0$ such that  $|\alpha|\neq 0$. Then there exists $C>0$  such that 
	$$\Vert f-\psi\ast f\Vert_{\mathcal{HM}_{q}^{\lambda}}\leq C\Vert f\Vert_{h\mathcal{M}_{q}^{\lambda}},$$
	for all $f\in h\mathcal{M}_{q}^{\lambda}(\R^{n})$ with  $0<q\leq 1$ and $q\leq\lambda<\infty$. 
\end{lemma}

\begin{proof}  Given $f\in h\mathcal{M}_{q}^{\lambda}(\R^n)$ we can estimate 
\small{\begin{align*}
\Vert f-\psi\ast f\Vert_{\mathcal{HM}_{q}^{\lambda}}
&\lesssim  \left\Vert \sup_{0<t\leq 1}\left\vert \varphi_t\ast f\right\vert\right\Vert_{\mathcal{M}_{q}^{\lambda}} +\;\;\left\Vert \sup_{0<t\leq 1}\left\vert (\varphi_t\ast \psi)\ast f\right\vert\right\Vert_{\mathcal{M}_{q}^{\lambda}}
+\;\,\left\Vert\sup_{t>1}\left\vert  (\varphi_t-\varphi_t\ast\psi)\ast f\right\vert\right\Vert_{\mathcal{M}_{q}^\lambda}\nonumber\\
&:= \Vert f\Vert_{h\mathcal{M}_q^{\lambda}} + I_1(f) +I_2(f).
\end{align*}}
Since  $\left\{ \varphi_{t} \ast \psi \right\}_{t \leq 1}$ and $\left\{ \varphi_{t} - \varphi_{t} \ast \psi \right\}_{t>1}$ are bounded on $\mathcal{S}(\R^{n})$ (see \cite[Lemma 4]{G}), then 
\begin{equation}
\vert (\varphi_{t} \ast \psi)\ast f\vert  \lesssim m_{\mathcal{F}}f\; \text{ for  } \; t\leq 1 \quad \text{ and } \quad |(\varphi_{t}-\varphi_{t} \ast \psi)\ast f\vert \lesssim m_{\mathcal{F}}f \;\text{ for } \; t>1, \nonumber
\end{equation} 
which yields $I_j(f)\lesssim  \|m_{\mathcal{F}}f\|_{\mathcal{M}_q^{\lambda}} \lesssim \|f\|_{h\mathcal{M}_q^{\lambda}}$ for $j=1,2$ as we wished to show. 
\end{proof}


\subsection{Some properties of Hardy-Morrey spaces}

The next two lemmas are well known and we given a proof just for read convenience.

\begin{lemma}\label{calor} Let $0<q \leq \lambda<\infty$ and $0<p \leq \gamma<\infty$ be such that $p/\gamma \,=\,q/\lambda$. If $f\in\mathcal{HM}_{q}^{\lambda}(\R^n)$ and $\lambda \leq  \gamma$ then there exists $C>0$ independent of $f$ such that 
	\begin{equation}\label{regular-estimate}
\|\varphi_{t}\ast f\|_{\mathcal{M}_{p}^{\gamma} }\leq C\, t^{n\left(\frac{1}{\gamma}-\frac{1}{\lambda}\right)} \|f \|_{\mathcal{HM}_{q}^{\lambda}}, \quad \quad  \; t>0,
\end{equation}
provided $\varphi\in\mathcal{S}(\mathbb{R}^n)$ and  $\int_{\mathbb{R}^n}\varphi(x) dx \neq 0$.
\end{lemma}

\begin{proof} 
By inequality \eqref{maximalop1} we can infer 
	\begin{equation} \label{control1}
	\|\varphi_t\ast f\|_{\mathcal{M}^{\lambda}_{q}}\leq \|M^{\ast}_{\varphi}f\|_{\mathcal{M}^{\lambda}_{q}} \lesssim \|f\|_{\mathcal{HM}^{\lambda}_{q}}, \quad \quad t>0. 
	\end{equation}
Now let $Q_t$ be a cube with sidelength $\ell_{Q_{t}}=t$, then 
	\begin{equation}\label{control2}
	\|\varphi_{t}\ast f\|_{\infty} \lesssim\left( \frac{1}{|Q_{t}|}\int_{Q_{t}}|M^{\ast}_{\varphi}f(y)|^{q}dy \right)^{1/q}\lesssim |Q_t|^{-\frac{1}{\lambda}}\|f\|_{\mathcal{HM}^{\lambda}_{q}}\lesssim t^{-\frac{n}{\lambda}}\|f\|_{\mathcal{HM}^{\lambda}_{q}},
	\end{equation}
where in the first inequality we used the pointwise control  $M_{\varphi}f(x)\leq M^{\ast}_{\varphi}f(x)\nonumber$ for all $|x-y|<t$. Now, from {convexity of Morrey spaces} {using}  \eqref{control1} and \eqref{control2} we obtain the desired estimate provided $\lambda \leq \gamma$ satisfy   $p\lambda=q\gamma$.
	\end{proof}

Taking $p=1$ at \eqref{regular-estimate} and recalling $\gamma=\lambda/q \geq 1$ we conclude 
\begin{equation} \label{2.8-heat}
\|\varphi_{t}\ast f\|_{\mathcal{M}_{1}^{\gamma} }\leq C\, t^{-\frac{n}{\lambda}(1-q)} \|f \|_{\mathcal{HM}_{q}^{\lambda}}, \quad \quad  \; t>0.
\end{equation} 
Analogous to the Hardy spaces, moment conditions play a fundamental role in the theory of Hardy-Morrey spaces. The Fourier transform decaying (\ref{decay-Fourier})  {implies} that restriction 
$
\int_{\R^{n}}f(x)dx=0
$  is necessary for a bounded, compactly supported function $f$ belong to Hardy-Morrey space $\mathcal{HM}^\lambda_{q}(\R^n)$ as $0<q\,\leq\lambda<1$. Bootstrapping the previous argument we may obtain 
 \begin{align}\label{moment}
\int_{\R^{n}} x^{\alpha}f(x)dx=0\; \text{ for }\; |\alpha|\leq \left\lfloor n\left({1}/{\lambda}-1\right)\right\rfloor
\end{align}
where $x^\alpha=x_1^{\alpha_1}\cdots x_n^{\alpha_n}$ for $\alpha \in \N^{n}_0$ and $\lfloor \cdot \rfloor$ denotes the floor function. This property {shows us} that  Hardy-Morrey spaces for $0<q\leq \lambda <  1$ are not closed by multiplication on test functions. Indeed, let  $f$ be as in Proposition \ref{prop2.5}. Then $\int_{\R^n}f(x)dx=0$. Hence, for all $\varphi \in C_{c}^{\infty}(\R^n)$ such  that $\varphi f \in \mathcal{HM}^\lambda_{q}(\R^n) $  we conclude that $\int_{\R^n}f(x)\varphi(x)dx=0$ which is a contradiction, since $f\not\equiv0$. The necessity of moment conditions for bounded and compactly supported functions on  $\mathcal{HM}^\lambda_{q}(\R^n)$ with  $0<q\leq 1$ and $q\leq\lambda$ were previously remarked in the works \cite{Tanaka,HH} in the setting of  atomic decomposition theorems (see for instance Theorem \ref{atomich} below). Next we present an alternative and elementary proof for this claim. 

\begin{proposition}\label{prop3.4} Let $0<q\leq 1$ and $q\leq\lambda<\infty$.  A bounded, compactly supported function $f$
satisfying the moment condition 
 \begin{align}\label{moment_2}
\int_{\R^{n}} x^{\alpha}f(x)dx=0\; \text{ for }\; |\alpha|\leq L \,\,\,\,  \text{with}\;\; L\geq 
N_{q}:=\left\lfloor n\left({1}/{q}-1\right)\right\rfloor,
\end{align}
belongs to $\mathcal{HM}_q^{\lambda}(\R^n)$ and moreover $\|f\|_{\mathcal{HM}_{q}^{\lambda}} \lesssim \|f\|_{L^{\infty}}|Q|^{1/ \lambda}$ for all cube $Q\supseteq \text{supp}(f)$. 
\end{proposition}


\begin{proof}Let  $f$ be a bounded function compactly supported in a cube $Q=Q(x_{Q},\ell_{Q})$ and let $Q^{\ast}=2Q$ be the double cube of $Q$ with sidelength  $\ell_{Q^{\ast}}=2\ell_Q$ and center $x_{Q^{\ast}}=x_Q$.  Let $J$ be a fixed cube, then for previous moment condition \eqref{moment_2} (see \cite[p.106]{S} 
we have
\begin{align}\label{2.8xx}
M_{\varphi}f(x)\lesssim \mathds{1}_{J\cap Q^{\ast}}(x) M_{\varphi}f(x)+ \mathds{1}_{J\backslash Q^{\ast}}(x)\|f\|_{L^{\infty}} \left(\frac{1}{1+\vert x-x_Q\vert/\ell_Q}\right)^{n+L+1},
\end{align}
where $\mathds{1}_A(x)$ denotes the indicator function of a subset $A$. Choosing a suitable $\varphi$ one has $M_{\varphi}f(x) \leq \| \varphi\|_{L^{1}}Mf(x)$ for $x\in \R^n$ almost everywhere, where $M$ denotes the Hardy-Littlewood maximal function which is bounded in $L^{r}(\R^{n})$ for $1<r \leq \infty$. It follows by H\"older inequality that 
\begin{align}
|J|^{q/\lambda-1} \int_{J \cap Q^{\ast}}\vert M_{\varphi}f(x)\vert^qdx 
&\leq |J|^{q/\lambda-1}\left(\int_{J \cap Q^{\ast}}\vert M_{\varphi}f(x)\vert^r dx\right)^{{{q}}/{r}} |J \cap Q^{\ast}|^{1-{q}/{r}}\nonumber\\
&\lesssim \left(\int_{\R^n}\vert f(x)\vert^r dx\right)^{{q}/{r}} |J|^{q/\lambda-1} |J \cap Q^{\ast}|^{1-{q}/{r}}   \nonumber\\
&\leq   \|f\|^{q}_{L^{\infty}} |Q|^{q/r}  |J|^{q/\lambda-1} |J \cap Q^{\ast}|^{1-{q}/{r}} \label{key-est-1} \\
&\lesssim  \|f\|^q_{L^{\infty}}  |Q|^{q/\lambda},\nonumber
\end{align}
where the last inequality will be justified as follows: if $|Q|<|J|$ then
$$  |Q|^{q/r}  |J|^{q/\lambda-1} |J \cap Q^{\ast}|^{1-{q}/{r}}  \leq  |Q|^{q/r}  |Q|^{q/\lambda-1} |Q^{\ast}|^{1-{q}/{r}} \lesssim |Q|^{q/\lambda}, $$
since $q/\lambda-1\leq 0$ and $1-q/r \geq 0$, however if $|J|\leq|Q|$ then
\begin{align*}
|Q|^{q/r}  |J|^{q/\lambda-1} |J \cap Q^{\ast}|^{1-{q}/{r}} 
=  \left(\frac{|J|}{|Q|} \right)^{q/\lambda-q/r} \left(\frac{|J \cap Q^{\ast}|}{|J|} \right)^{1-q/r}  |Q|^{q/\lambda}  
 \leq  |Q|^{q/\lambda}, 
\end{align*}
where we choose $r>\lambda $.  

{Since $x\in J\backslash Q^{\ast}$ satisfies $\vert x-x_Q\vert\geq 2\ell_{Q}$}, from \eqref{2.8xx} and assuming  $\vert J\vert >\vert Q\vert$ we have, 
\begin{align}
\vert J\vert^{{q}/{\lambda}-1}\int_{J\backslash Q^{\ast}}|M_{\varphi}(f)(x)|^{q}  dx
	&\lesssim  \vert J\vert^{{q}/{\lambda}-1} \| f \|^{q}_{L^\infty} \int_{\vert x-x_Q\vert \geq 2\ell_{Q}}\left(1+\frac{\vert x-x_Q\vert}{\ell_Q}\right)^{-(n+L+1)q}dx \nonumber\\
	&\lesssim \vert J\vert^{{q}/{\lambda}-1} \| f \|^{q}_{L^\infty} \ell_Q^{(n+L+1)q} \int_{2\ell_{Q}}^{\infty}r^{-q(n+L+1)+n-1}dr\nonumber\\
&\lesssim\| f \|^{q}_{L^\infty}  \vert Q\vert^{{q}/{\lambda}},  \nonumber 
\end{align}
in view of $L\geq \left\lfloor n\left({1}/{q}-1\right)\right\rfloor$. The case  $\vert J\vert \leq \vert Q\vert$ we  easily have the estimate
	\begin{align}
	\vert J\vert^{{q}/{\lambda}-1} \int_{J\backslash Q^{\ast}}|M_{\varphi}(f)(x)|^{q} dx 	&\lesssim  \vert J\vert^{{q}/{\lambda}-1} \|f\|^{q}_{L^{\infty}}  \int_{J\backslash Q^{\ast}}\left(1+\frac{\vert x-x_Q\vert}{\ell_Q}\right)^{-(n+L+1)q}dx  \nonumber\\
	&\lesssim \|f\|^{q}_{L^{\infty}}   \vert J\vert^{{q}/{\lambda}-1}  \vert J\backslash Q^{\ast} \vert \nonumber \\
	&\lesssim \|f\|^{q}_{L^{\infty}} |Q|^{q/\lambda}. \nonumber
	\end{align}
Combining the previous controls, we obtain  
$\|f\|_{\mathcal{HM}_{q}^{\lambda}} \lesssim \|f\|_{L^{\infty}}|Q|^{1/ \lambda}$. In particular,  $f \in \mathcal{HM}_q^{\lambda}(\R^n)$.
\end{proof}

\subsection{Atomic decomposition  in Hardy-Morrey spaces}
\begin{definition}[\cite{HH}]\label{Hatom}  Let $0<q\leq 1$ and  $q \leq \lambda<\infty$. A function $a$   mensurable is called {\it\textnormal a} bounded  $\mathcal{HM}_q^{\lambda}(\mathbb{R}^{n})-$ atom or  $(q,\lambda,\infty)-$atom if 
	\begin{equation}
		(i)\;\text{supp}(a)\subset Q\quad\quad  (ii)\;\Vert a\Vert_{L^\infty}\leq \vert Q\vert ^{-{1}/{\lambda}} \quad \text{ and }\quad (iii)\;\int_{\mathbb{R}^n} x^{\alpha}a(x)dx =0, \nonumber
	\end{equation}
	for multi-index $\alpha\in\mathbb{N}_0^n$ with  $\vert\alpha\vert\leq  N_{q}$.
\end{definition}

To fix the dependence on support of $a$ {we write} $a=a_{Q}$ for every $(q,\lambda,\infty)-$atom supported in a cube $Q$. It follows directly from Proposition \ref{prop3.4} that $\|a_{Q}\|_{\mathcal{HM}_q^{\lambda}}\lesssim 1$.

We denote the atomic space $\textbf{at}\mathcal{HM}_{q}^{\lambda}(\R^n)$, also called atomic space of $\mathcal{HM}_{q}^{\lambda}(\R^n)$,  by the collection of  distributions $f\in\mathcal{S}'(\mathbb{R}^n)$ such that 
\begin{equation}\nonumber 
f =\sum_{Q\,:\, \text{dyadic}}s_{Q}a_{Q} \,\, \text{ in } \,\, {\mathcal{S}'(\mathbb{R}^n)},
\end{equation}
where $\{a_Q\}_Q$ are  $(q,\lambda,\infty)$-atoms and $\{s_Q\}_{Q}$ are complex scalars satisfying  
\begin{equation}\nonumber
\Vert \{s_Q\}_Q\Vert_{{\lambda,q}}:=\sup_{J} \left\{\left({\vert J\vert^{q/\lambda-1}} \sum_{\substack{Q\subseteq J}}\left(\vert Q\vert^{{1}/{q}-{1}/{\lambda}}\,\vert s_{Q}\vert\right)^q \right)^{{1}/{q}}\right\}<\infty.
\end{equation}
The functional 
\begin{equation}
\Vert f \Vert_{\textbf{at}\mathcal{HM}_{q}^{\lambda}}:= \inf \left\{  \Vert \{s_Q\}_{Q}\Vert_{{\lambda,q}} : f =\sum_{Q}s_{Q}a_{Q} \right\}\nonumber
\end{equation}
defines a {quasi-norm} for the space usually called  {quasi-atomic norm}. 
 Jia and Wang stated the following atomic decomposition theorem.
\begin{theorem}[\cite{HH}]\label{atomich}
	Let $0< q \leq1$, $q<\lambda<\infty$ and $f\in \mathcal{HM}_q^{\lambda}(\R^{n})$. Then there is a sequence $\{a_Q:\;  Q\text{ dyadic}\}$ of $(q,\lambda,\infty)-$atoms and $\{s_{Q}\}_{Q}$ complex numbers   such that 
	\begin{equation}\label{conver}
	f =\sum_{Q\,:\, \text{dyadic}}s_{Q}a_{Q} \;\text { in }\; \mathcal{S}'(\mathbb{R}^n)\; \text{ and }\; \Vert f \Vert_{\textbf{at}\mathcal{HM}_{q}^{\lambda}}\lesssim \Vert f\Vert_{\mathcal{HM}_q^\lambda}.
	\end{equation} 
Conversely, given a sequence of 
 $(q,\lambda,\infty)-$atoms $\{a_Q:\; Q \text{ dyadic}\}$ and  
 complex numbers $\{s_{Q}\}_{Q}$ such that $f =\sum_{Q}s_{Q}a_{Q}$ in  $\mathcal{S}'(\mathbb{R}^n)$  satisfy {$\Vert \left\{ s_{Q} \right\}_Q \Vert_{\lambda,q}<\infty\,$} 
then $f \in \mathcal{HM}_q^\lambda(\R^n)$ and 
$\left\Vert f\right\Vert_{\mathcal{HM}_q^\lambda} \lesssim \Vert f \Vert_{\textbf{at}\mathcal{HM}_{q}^{\lambda}}$.
\end{theorem}

The previous statement recover the atomic decomposition for Hardy spaces (i.e. $q=\lambda$) with peculiar difference that the decomposition at \eqref{conver} in general does not converge in $\mathcal{HM}_q^\lambda(\R^n)$ norm (see \cite[Remark 3.6]{HH}). 
In particular, several properties obtained uniformly by $(q,\lambda,\infty)-$atoms can not be extended easily for $\mathcal{HM}_{q}^{\lambda}(\R^n)$. 

\begin{definition}\label{defatom}Let $0<q\leq 1$ and $ q \leq \lambda<\infty$. A  measurable function $a_Q$ supported in a cube  $Q$ is called a bounded $h\mathcal{M}_q^{\lambda}-$atom, if it  satisfies 
$ \Vert a_Q\Vert_{L^\infty}\leq \vert Q\vert ^{-{1}/{\lambda}} $ and 
\begin{equation}\label{cond-hM-atom}
\int_{\mathbb{R}^n} a_Q(x)x^{\alpha}dx =0
\end{equation}
for all dyadic cube $|Q|<1$ and multi-index $\alpha\in\mathbb{N}_0^n$ with  $\vert\alpha\vert\leq  N_{q}$. 
\end{definition}

In the previous definition  we require vanish moments only for atoms supported on small cubes $Q$ satisfying $|Q|<1$, the unique difference with $\mathcal{HM}_q^\lambda-$atoms. We say that $a_{Q}$ is a rough $h\mathcal{M}_q^{\lambda}-$atom or \textit{$(\infty,\lambda)-$block}, if $a_{Q}$ is a measurable function supported in the cube $Q$ with sidelength $\ell_Q\geq 1$ and $\Vert a_{Q}\Vert_{L^\infty(\R^n)}\leq \vert Q\vert^{-1/\lambda}$.

\begin{lemma}\label{lemmaatomic} If $a_{Q}$ is {an} $h\mathcal{M}_q^{\lambda}-$atom, then $\|a_Q\|_{h\mathcal{M}_q^{\lambda}} \lesssim 1$ with implicit constant independent of $a_{Q}$.
\end{lemma}

\begin{proof} Let $a_{Q}$ be an $h\mathcal{M}_q^{\lambda}-$atom supported in a cube $Q$. From Proposition \ref{prop3.4} it is sufficient consider that $a_{Q}$ is a  rough $h\mathcal{M}_q^{\lambda}-$atom. Let $B=B(x_{0},R)$ be a ball that contains Q.
Consider $\varphi \in C_{c}^{\infty}(B(0,1))$  nonnegative such that $\|\varphi\|_{L^{1}}=1$ then $\text{supp}{(\varphi_{t} \ast a_{Q})} \subseteq B(x_0,R+1)$ as $0<t\leq 1$. Since $R>1$ then $B(x_0,R+1)\subset B^{\ast}$ and for a fixed cube $J$ we have
\begin{align*}
|J|^{q/\lambda-1}\int_{J}|m_{\varphi}a_{Q}(x)|^{q}dx&\leq  |J|^{q/\lambda-1}\int_{J \cap B^{\ast}}|m_{\varphi}a_{Q}(x)|^{q}dx \\
&\leq |J|^{q/\lambda-1}  \left(\int_{\R^{n}}|m_{\varphi}a_{Q}(x)|^{r} dx\right)^{q/r}|J \cap B^\ast|^{1-q/r}\\
& \lesssim \|a_{Q}\|^{q}_{L^{r}} |J|^{q/\lambda-1}|J \cap B^\ast|^{1-q/r}\lesssim 1,
\end{align*}
where the last inequality follows from (\ref{key-est-1}).
\end{proof}

We denote the atomic space {$\textbf{at}h\mathcal{M}_{q}^{\lambda}(\R^n)$} 
by the collection of  distributions $f\in\mathcal{S}'(\mathbb{R}^n)$ 
such that 
\begin{equation}
f =\sum_{Q}s_{Q}a_{Q} \text{ in } \mathcal{S}'(\mathbb{R}^n),\nonumber
\end{equation}
where $\{a_Q\}_{Q}$ are  bounded $h\mathcal{M}_q^{\lambda}-$atoms and $\{s_Q\}$ are complex scalar satisfying  
\begin{equation}\label{eq4.5}
\Vert \{s_Q\}_{Q}\Vert_{{\lambda,q}}:=\sup_{J} \left\{\left({\vert J\vert^{q/\lambda-1}} \sum_{\substack{Q\subseteq J}}\left(\vert Q\vert^{{1}/{q}-{1}/{\lambda}}\,\vert s_{Q}\vert\right)^q \right)^{{1}/{q}}\right\}<\infty.
\end{equation}
The functional 
\begin{equation}\nonumber
\Vert f \Vert_{\textbf{at}h\mathcal{M}_{q}^{\lambda}}:= \inf \left\{  \Vert \{s_Q\}_{Q}\Vert_{{\lambda,q}} : f =\sum_{Q}s_{Q}a_{Q} \right\}
\end{equation}
defines a {quasi-norm} for the space usually called {quasi-atomic norm}. Now we ready to announce an atomic decomposition for localizable Hardy-Morrey space $h\mathcal{M}_q^\lambda(\R^n)$.  

\begin{theorem}\label{atomich22}
	Let $0< q \leq1$, $q<\lambda<\infty$ and $f\in h\mathcal{M}_q^{\lambda}(\R^n)$. Then there is a sequence of bounded $h\mathcal{M}_q^{\lambda}-$atoms $\{a_Q: Q \; \text{dyadic}\}$ and a sequence of {complex scalars}   $\{s_{Q}\}_{Q}$ such that 
	\begin{equation} 
	f =\sum_{Q\,:\,dyadic} s_{Q}a_{Q} \;\text { in }\; \mathcal{S}'(\mathbb{R}^n)\; \text{ and }\; \Vert \{s_Q\}_{Q}\Vert_{{\lambda,q}}\lesssim \Vert f\Vert_{h\mathcal{M}_q^\lambda}.
	\end{equation} 
Conversely, for every sequence $\{a_Q\}_{Q}$ of bounded  $h\mathcal{M}_q^{\lambda}-$atoms and scalars 
 $\{s_{Q}\}_{Q}$ satisfying $\Vert \{s_Q\}_{Q}\Vert_{\lambda,q}<\infty$, then ${f=\sum_Q s_{Q}a_{Q} \in h\mathcal{M}_q^\lambda}(\R^n)$ and 
$\left\Vert f\right\Vert_{h\mathcal{M}_q^\lambda} \lesssim \Vert \{s_Q\}_{Q}\Vert_{{\lambda,q}}$.
\end{theorem}

\begin{proof} 

Let us start with the second part. Given sequence $\{a_Q\}_{Q}$  of bounded  $h\mathcal{M}_q^{\lambda}-$atoms and $\{s_{Q}\}_{Q}$ satisfying \eqref{eq4.5} we split
\begin{equation}\nonumber
f =\sum_{Q\,:\, \text{dyadic}}s_{Q}a_{Q}={\sum_{{Q, \;\ell_{Q}< 1}}s_{1,Q}\,a_{1,Q}+\sum_{{Q, \;\ell_{Q}\geq 1}}s_{2,Q}\,a_{2,Q}}:=f_{1}+f_{2},
\end{equation}
where $\{a_{1,Q}\}_{Q}$ are $\mathcal{HM}^{\lambda}_{q}-$atoms and $\{a_{2,Q}\}_{Q}$ are rough $h\mathcal{M}^{\lambda}_{q}-$atoms. Clearly 
\begin{equation}\label{quebra}
\Vert \{s_{1,Q}\}_{Q}\Vert_{{\lambda,q}}+ \Vert \{s_{2,Q}\}_{Q}\Vert_{{\lambda,q}} \simeq \Vert \{s_Q\}_{Q}\Vert_{{\lambda,q}}.
\end{equation}
From atomic decomposition Theorem \ref{atomich},  we have $f_{1} \in \mathcal{HM}^{\lambda}_{q}$ and 
\begin{equation}\nonumber
\left\Vert f_1\right\Vert_{h\mathcal{M}_q^\lambda} \leq \left\Vert f_{1}\right\Vert_{\mathcal{HM}_q^\lambda} \lesssim \Vert \{s_{1,Q}\}_{Q}\Vert_{{\lambda,q}}.
\end{equation}
Now fix a dyadic cube $J$ and assume that $f_{2}= \sum_{Q}s_{2,Q}a_{2,Q}$ {that converges as distribution in $\mathcal{S}'(\mathbb{R}^n)$,} 
then 
\begin{align*}
\Big\Vert m_{\varphi}{\Big(\sum_{Q}s_{2,Q}a_{2,Q}\Big)}\Big\Vert^{q}_{L^{q}(J)} \leq
 \sum_{Q} |s_{2,Q}|^{q} \int_{J \cap Q^{\ast}}|m_{\varphi}(a_{2,Q})|^{q}dx 
& \leq \sum_{Q} |s_{2,Q}|^{q}  \|a_{2,Q}\|^{q}_{L^{\infty}} |J \cap Q^{\ast}| \\
& \leq \sum_{Q} |s_{2,Q}|^{q}  |Q|^{-q/\lambda} |J \cap Q^{\ast}|,
\end{align*}
where $\text{supp}(m_{\varphi}(a_{2,Q})) \subseteq Q^{\ast}=2Q$, since {$\ell_Q\geq 1$}. Fixed a cube $J$ we split the previous sum into  $\sum_{Q}=\sum_{Q\subseteq J}+\sum_{J\subset Q}$.  For the first sum we have 
\begin{equation}\nonumber
 \sum_{Q \subseteq J} |s_{2,Q}|^{q}  |Q|^{-q/\lambda} |J \cap Q^{\ast}| \lesssim  \sum_{Q \subseteq J} |s_{2,Q}|^{q}  |Q|^{1-q/\lambda}\lesssim \vert J\vert ^{1 - q/\lambda} \Vert \{s_{2,Q}\}_Q\Vert_{{\lambda,q}}^q
\end{equation}
and the second sum is bounded by 
\begin{equation}\nonumber
 \sum_{J \subset Q} |s_{2,Q}|^{q}  |Q|^{-q/\lambda} |J \cap Q^{\ast}| \leq  |J|^{1-q/\lambda} \sum_{J \subset Q} |s_{2,Q}|^{q}  (\vert J\vert/\vert Q\vert)^{q/\lambda}.
\end{equation}
Fixed $J$ dyadic  there exists a subset $\tilde{N} \subseteq \N$ such that each cube $J \subset Q$ associated to rough atom in the decomposition given of $f_{2}$ is uniquely determined by $Q:=Q_{k,J} $ for $k \in \tilde{N}$ such that $\ell_{Q_{k,J}}=2^{k}\ell_{J}$. Then
\begin{align}
 \sum_{J \subset Q} |s_{2,Q}|^{q}  (\vert J\vert / \vert Q\vert)^{q/\lambda} &=  \sum_{k \in \tilde{N}} |s_{2,Q_{k,J}}|^{q}  2^{-knq/\lambda} \nonumber\\
 & \leq \sup_{k \in \tilde{N}} \left\{ |s_{2,Q_{k,J}}|^{q} \right\} \sum_{k} 2^{-knq/\lambda}\nonumber\\
  & \lesssim \sup_{k \in \tilde{N}} \left\{ |Q_{k,J}|^{q/\lambda-1}  |Q_{k,J}|^{1-q/\lambda} |s_{2,Q_{k,J}}|^{q} \right\} \nonumber\\
& \leq  \sup_{k \in \tilde{N}} \left\{ |Q_{k,J}|^{q/\lambda-1}  \sum_{\widetilde{Q}\, \subseteq\, Q_{k,J}} \left(|\widetilde{Q}|^{1/q-1/\lambda}|s_{2,\widetilde{Q}}|\right)^{q} \right\}\nonumber\\
&\lesssim  \|\{s_{2,Q}\}_{Q}\|_{\lambda,q}^{q}\nonumber
\end{align}
uniformly on $J$, where $\widetilde{Q} \subset Q_{k,J}$ in the third inequality is taken on the family of dyadic cubes generators of $f_{2}$.  Combining the previous controls we obtain
\begin{align*}
	\Vert f_{2} \Vert_{h\mathcal{M}_{q}^{\lambda}}
	& \simeq \sup_{J} \left\{ \left( |J|^{{q}/{\lambda}-1} \Big\Vert m_{\varphi}{\Big(\sum_{Q}s_{2,Q}a_{2,Q} \Big)}\Big\Vert^{q}_{L^{q}(J)} \right)^{1/q} \right\}\\
	&\lesssim \sup_{J} \left\{ \left( |J|^{{q}/{\lambda}-1}\sum_{Q }|s_{2,Q}|^{q}|Q|^{-{q}/{\lambda}}|J \cap Q^{\ast}| \right)^{1/q} \right\}\\
	& \lesssim \|\{s_{2,Q}\}_{Q}\|_{\lambda,q}.
	\end{align*}
The conclusion follows from \eqref{quebra}. 
 
Let us moving on to the first part. Suppose $f\in h\mathcal{M}_q^{\lambda}(\R^n)$ and let  $\psi\in\mathcal{S}(\mathbb{R}^n)$ be such that $\int_{\mathbb{R}^n} \psi(x) dx =1$ and  $\int_{\mathbb{R}^n} x^{\alpha}\psi(x) dx =0$ for $0<\vert\alpha\vert\leq N_{q}$. From Lemma \ref{lem-glob_loc}, the map $f\in h\mathcal{M}_q^{\lambda}(\R^n)\mapsto (f-\psi\ast f)\in \mathcal{HM}_{q}^\lambda(\R^n)$ is continuous and thus by Theorem \ref{atomich} we can write  
\begin{align}\label{key2}
	f-\psi\ast f=\sum_{Q\,:\, \text{dyadic}} s_Qa_Q \;\text{ in }\; \mathcal{S}'(\mathbb{R}^n),
\end{align}
where $\{a_Q\}_{Q}$ are $\mathcal{HM}_{q}^{\lambda}$-atoms and the scalars $\{s_Q\}_{Q}$ satisfies $\left\Vert \{s_Q\}_{Q}\right\Vert_{\lambda,q}\lesssim \Vert f-\psi\ast f\Vert_{\mathcal{HM}^{\lambda}_q}\lesssim \Vert f\Vert_{h\mathcal{M}_q^{\lambda}}$.
Now we will decompose $\psi\ast f$ into rough atoms. Indeed, let us write  
\begin{equation}\label{decom-b_Q}
\psi\ast f= \sum_{Q} (\psi\ast f)\mathds{1}_{Q}\,:=\sum_{Q} b_{Q},
\end{equation} 
where $\left\{Q\right\}$ {is the family of unit cubes given by translation of cube $[0,1[^n$, thus a partition of $\R^n$}. 

Given $x\in Q$, the inequality $\vert (\psi\ast f)(x)\vert \leq m^{\ast}_\psi f(y)$ holds for every $y$ satisfying  $|x-y|< t\leq 1$ then
\begin{equation}
\vert b_Q(x)\vert \leq \left(\frac{1}{|Q|}\int_{Q} \left\vert m_{\psi}^{\ast}f(y)\right\vert^qdy\right)^{{1}/{q}}=\; \vert Q\vert^{-1/q}\|m_{\psi}^{\ast} f\|_{L^{q}(Q)}.\nonumber
\end{equation}
For each dyadic cube $Q$ consider 
	\begin{equation}
	\tilde{b}_Q= \frac{\vert Q\vert^{{1/q -1/\lambda}}}{\|m_{\psi}^{\ast} f\|_{L^{q}(Q)}}b_{Q}\nonumber
	\end{equation}
if denominator is not zero or  $\tilde{b}_Q=0$ otherwise. Hence, $\{\tilde{b}_Q\}_Q$ is a sequence of rough atoms, because $\text{supp}\,(\tilde{b}_Q)\subseteq Q$ and $\Vert \tilde{b}_Q\Vert_{\infty} \leq \vert Q\vert^{-{1}/{\lambda}}$. If we rewriting (\ref{decom-b_Q}) as follows 
	\begin{equation}\label{key1}
\psi\ast f= \sum_{Q}\tilde{s}_{Q}\tilde{b}_{Q},
\end{equation} 
where $\tilde{s}_Q= \vert Q\vert^{1/\lambda -1/q}\|m^{\ast}_{\psi}f\|_{L^{q}(Q)}$. Then,
\small{\begin{align}
\left\Vert \{\tilde{s}_Q\}_{Q}\right\Vert_{\lambda,q}
= \sup_{J} \left\{\left({\vert J\vert^{\frac{q}{\lambda}-1}} \sum_{Q\subseteq J}\left(\vert Q\vert^{{1}/{q}-{1}/{\lambda}}\,\vert \tilde{s}_{Q}\vert\right)^q \right)^{\frac{1}{q}}\right\}\nonumber
&= \sup_{J} \left\{ \left(\vert J\vert^{\frac{q}{\lambda}-1}\sum_{Q\subseteq J} \int_Q\vert  m_{\psi}^{\ast}f\vert^q dy\right)^{\frac{1}{q}} \right\} \nonumber\\
&\leq \sup_{J} \left\{   \left(\vert J\vert^{{q}/{\lambda}-1}  \int_J\vert m_{\psi}^{\ast}f(y)\vert^q dy\right)^{\frac{1}{q}} \right\} \nonumber\\
&\lesssim \Vert f\Vert_{h\mathcal{M}_q^{\lambda}} \nonumber.
\end{align}}
It follows from (\ref{key2}) and (\ref{key1}) that  
\begin{equation}\label{XX}
f=\sum_{Q\,:\, \text{dyadic}} s_Qa_Q+\sum_{Q\,:\, \text{dyadic}} \tilde{s}_{Q}\tilde{b}_{Q}\;  \text{ in }\mathcal{S}'(\R^n)
\end{equation}
 is an atomic decomposition for $f\in h\mathcal{M}_q^\lambda(\R^n)$ such that $\left\Vert \{s_Q\}_{Q}\right\Vert_{\lambda,q}+\left\Vert \{\tilde{s}_Q\}_{Q}\right\Vert_{\lambda,q}\lesssim\Vert f\Vert_{h\mathcal{M}_q^{\lambda}}$.
\end{proof}

\begin{remark}{According to \cite{Sawano, Sawano2}, we have  $h\mathcal{M}_{q}^{\lambda}(\R^n) = \mathcal{E}^{0}_{\lambda 2q}(\R^n) = F^{0, 1/q-1/\lambda}_{q,2}(\R^n)$. It follows from  \cite[Theorem  4.12]{Sawano3} that distributions $f\in h\mathcal{M}_{q}^{\lambda}(\R^n)$ have an atomic decomposition  $f=\sum_{Q} s_{Q}a_{Q}$, where the atoms $\{a_{Q}\}$ can be chosen in $L^{\infty}(\mathbb{R}^n)$. However, the atomic space provides a different way to estimate the coefficients of the atomic decomposition $f=\sum_{Q}s_{Q}a_{Q}$.}
\end{remark}

\section{Pointwise Fourier transform decay in  $\mathcal{HM}_{q}^{\lambda}$}\label{Point-parte1}
{We start presenting a special function $\varphi \in  C_{c}^{\infty}(\mathbb{R}^n) $ to be considered at the next proposition. }
\begin{lemma}\label{suitable-cutoff}There exists $\varphi \in C_{c}^{\infty}(\mathbb{R}^n)$ with $\text{supp}(\varphi)\subset B(0,1)$ and nonnegative Fourier transform such that $\widehat{\varphi}(\xi)\geq 1$ for $|\xi| \leq 1$, where the Fourier transform is defined as 
$
		\widehat{\varphi}(\xi):=\int_{\mathbb{R}^n}e^{-ix\cdot \xi}\varphi(x)dx.\nonumber
$ 
\end{lemma}
\begin{proof}
Let $\eta \in C_{c}^{\infty}(\mathbb{R}^n)$ be a nonnegative  real valued function with $\text{supp}(\eta) \subseteq B(0,1/2)$. Then $ \psi:=\eta\ast \eta$ is supported on $B(0,1)$ and has nonnegative Fourier transform. Since $\widehat{\psi}(0)=(\int \eta(x)dx)^2>0$, by continuity we can choose $c_1>1$ and $c_2>0$ sufficiently large such that  $\displaystyle \min_{ \vert\xi\vert\leq 1} \{c_2\widehat{\psi}(\xi/c_1)\}\geq 1$. Hence $\varphi(x):= c_1^{n}c_2\psi (c_1x)$ has  the desired properties. 
\end{proof}

\begin{proposition}\label{prop2.5}
Let $0<q \leq \lambda\leq 1$ and $a_Q$ an $(q,\lambda,\infty)$-atom,  
then  $\widehat{a_Q}(\xi)$ is continuous and satisfy 
		\begin{equation}\label{decay-Fourier}
			\vert  \widehat{a_Q}(\xi) \vert  \lesssim |\xi|^{n\left(\frac{1}{\lambda}-1 \right)} \|a_Q \|_{\mathcal{HM}_{q}^{\lambda}},  {\quad \quad \quad  \forall  \,\,\xi \in \R^n}
\end{equation}
with implicit constant independent of $a_Q$.
\end{proposition}
\begin{proof}
Let $f=a_Q$ be a function compactly supported on a cube $Q=Q(x_0,\ell_{Q})$ centered at $x_0$ with sidelength $\ell_{Q}>0$ and let $\varphi \in C_{c}^{\infty}(\R^n)$ be according to Lemma \ref{suitable-cutoff}. Let $t \geq \ell_{Q}$ and set $\gamma=\lambda/q$. As $\text{supp }(\varphi_{t} \ast f) \subset Q(x_{0},\ell_{Q}+t)$ then we estimate 
\begin{align}\nonumber
\| \varphi_t\ast f \|_{L^{1}} 
\leq \vert Q(x_0,2t)\vert^{1-\frac{1}{\gamma}}\Vert  \varphi_t\ast f \Vert_{\mathcal{M}^{\gamma}_{1}}
\lesssim  t^{n\left(1-\frac{1}{\gamma} \right)} t^{-n\left(\frac{1}{\lambda}-\frac{1}{\gamma} \right)} \|f \|_{\mathcal{HM}_{q}^{\lambda}} 
=t^{n\left(1-\frac{1}{\lambda}\right)}\|f \|_{\mathcal{HM}_{q}^{\lambda}},
\end{align}
where the second inequality follows from \eqref{2.8-heat}. Since $\widehat{\varphi}(\zeta)\geq 1$ on sphere $\mathbb{S}^{n-1}$, then taking $t=\vert \xi\vert^{-1}$ we have
	\begin{equation}\label{key-est-for -hM}
			\vert \widehat{f}(\xi)\vert \leq \vert\widehat{\varphi}(\xi/\vert\xi\vert)\, \widehat{f}(\xi)\vert = \vert (\varphi_{\vert\xi\vert^{-1}}\ast f)^{\wedge}(\xi)\vert \leq \Vert \varphi_{\vert\xi\vert^{-1}}\ast f\Vert_{L^{1}} \lesssim  \vert \xi\vert^{n\left(\frac{1}{\lambda}-1 \right)} \|f \|_{\mathcal{HM}_{q}^{\lambda}} ,
	\end{equation}
if provided that $0<|\xi| \leq \ell_{Q}^{-1}$. Now  for $t=\vert \xi\vert^{-1}<\ell_{Q}\,$, it  follows that 
$\varphi_{t} \ast f$ is supported at $2Q=Q(x_{0},2\ell_{Q})$ and mimicking (\ref{control2})  with $2Q$ in place of $Q_t$ 
one has $\Vert \varphi_{t} \ast f\Vert_{L^\infty}\lesssim \ell_{Q}^{-n/\lambda}\Vert f\Vert_{\mathcal{HM}_{q}^{\lambda}}$. Hence, we may estimate
\begin{align*}
|\hat{f}(\xi)| \leq \Vert \varphi_{\vert\xi\vert^{-1}}\ast f\Vert_{L^{1}}\lesssim \vert Q\vert \,\Vert \varphi_{\vert\xi\vert^{-1}}\ast f\Vert_{L^{\infty}}&\lesssim \ell_{Q}^{n\left(1-\frac{1}{\lambda}\right)} \Vert f\Vert_{\mathcal{HM}_{q}^{\lambda}}
\lesssim \vert \xi\vert^{n\left(\frac{1}{\lambda}-1 \right)} \|f \|_{\mathcal{HM}_{q}^{\lambda}},
\end{align*} 
if provided $0<\lambda\leq 1$ . The case $\xi=0$ follows by continuity of Fourier transform.
\end{proof}


\begin{theorem}\label{cor-decay}
If  $f\in\mathcal{HM}_{q}^{\lambda}(\R^n)$ for {$0<q \leq \lambda \leq 1$},
then $\widehat{f}(\xi)$ is a continuous function 
and satisfy the {following Fourier transform decay} 
\begin{equation}\label{decay-Fourier2}
|\widehat{f}(\xi) | \lesssim  |\xi|^{n\left(\frac{1}{\lambda}-1 \right)} \|f \|_{\mathcal{HM}_{q}^{\lambda}}, {\quad  \forall  \,\,\xi \in \R^n}
\end{equation}
with implicit positive constant independent of $f$. 
\end{theorem}
\begin{proof} {We only have to assume $0<q<\lambda \leq 1$, since the conclusion holds for $q=\lambda$ i.e. Hardy spaces}. The Let $f=\sum s_{Q}a_{Q}$ be an atomic decomposition such that $\|\left\{s_{Q}\right\}_{Q} \|_{\lambda,q}
\lesssim  \|f \|_{\mathcal{HM}_{q}^{\lambda}}$ from Theorem \ref{atomich}.  We claim that there exists $C>0$ {independent of the decomposition} such that 
\begin{align}\label{convergente}
|J|^{\frac{q}{\lambda} -1} \sum_{Q\,:\, \text{dyadic}} |s_{Q}|^q  \|M_{\varphi}a_{Q}\|^q_{L^{q}(J)} ,
\leq C \|\left\{s_{Q}\right\}_{Q} \|_{\lambda,q}^q
\end{align}
uniformly on $J$ cube. 

Let us assume first the validity of \eqref{convergente} and consider $\left\{ Q_{j} \right\}_{j}$ an enumeration of the family $\left\{ Q\right\}$ associated to the decomposition $f=\sum s_{Q}a_{Q}$. From \eqref{convergente} for each $\epsilon>0$ there exists $\ell \in \N$ such that 
\begin{align}\label{convergente2}
\sum_{j=\ell}^{m} |s_{Q_{j}}|^q |J|^{\frac{q}{\lambda} -1} \|M_{\varphi}a_{Q_{j}}\|^q_{L^{q}(J)} 
\leq  C \epsilon^q 
\end{align}
with constant $C>0$ uniformly on cube $J$ and independent of the sum, whenever $m\geq\ell$. Note that each $(q,\lambda,\infty)$-atom $a_{Q}$, by Proposition \ref{prop2.5}, satisfies $\widehat{a_{Q}}$ is continuous and $\vert \widehat{a_{Q}}(\xi) \vert  \leq C {\vert \xi\vert ^{n\left({1}/{\lambda}-1 \right)} \Vert a_Q\Vert_{\mathcal{HM}_q^{\lambda}}}$ with constant independent of $a_Q$. Thereby,  combining the previous control 
with \eqref{convergente2},  
 for each $\epsilon>0$ small enough there exists $\ell \in \N$ such that 
\begin{align}
\left\vert \sum_{j=\ell}^m s_{Q_{j}} \widehat{a_{Q_{j}}}(\xi)\right\vert^q=\left\vert\Big(  \sum_{j=\ell}^m s_{Q_j}a_{Q_j} \Big)^{\widehat{\,\,\,}}(\xi)\,\right\vert^q & \lesssim \,\vert \xi\vert^{n q\left( \frac{1}{\lambda}-1 \right)}\Big\|\sum_{j=\ell}^m s_{Q_j}a_{Q_j} \Big\|^{q}_{\mathcal{HM}_{q}^{\lambda}}\nonumber\\
& \lesssim\,\vert \xi\vert^{n q \left( \frac{1}{\lambda}-1 \right)} \sup_{J} \left\{ |J|^{\frac{q}{\lambda} -1} \Big\|M_{\varphi}\Big( \sum_{j=\ell}^m s_{Q_j}a_{Q_j} \Big)\Big\|^q_{L^{q}(J)} \right\} \nonumber\\
& \leq \;\vert \xi\vert^{nq \left( \frac{1}{\lambda}-1 \right)} \sup_{J} \left\{ |J|^{\frac{q}{\lambda} -1} \sum_{j=\ell}^m |s_{Q_j}|^q\|M_{\varphi}a_{Q_j}\|^q_{L^{q}(J)} \right\} \label{p3}\\
& \leq C\,\varepsilon^q  \vert \xi\vert ^{nq \left( \frac{1}{\lambda}-1 \right)},\nonumber
\end{align}
whenever $m\geq\ell$, where was invoked $\vert \sum_{j}^m s_j\vert ^{q}\leq \sum_{j}^m\vert s_j\vert^{q}$ for  $0<q\leq 1$ in the third inequality. 
Thus $\sum_Q s_{Q}\widehat{a_{Q}}$ is a Cauchy sequence on $C(K)$ - the set of continuous function on compact set $K \subset \subset \R^n$ - that is complete and then the series converges uniformly to a continuous function on $C(K)$.  Moreover, following the same steeps in the proof of \eqref{p3} we have
\begin{equation}\nonumber
\big\vert \sum_{j=1}^m s_{Q_{j}} \widehat{a_{Q_{j}}}(\xi)\big\vert 
 \leq \;\vert \xi\vert^{n \left( \frac{1}{\lambda}-1 \right)} \sup_{J} \left\{ |J|^{\frac{q}{\lambda} -1} \sum_{j=\ell}^m |s_{Q_j}|^q\|M_{\varphi}a_{Q_j}\|^q_{L^{q}(J)} \right\} 
\leq C\vert \xi\vert ^{n \left(\frac{1}{\lambda}-1 \right)} \|\left\{s_{Q}\right\}_Q \|_{\lambda,q},
\end{equation}
where the last inequality we use \eqref{convergente}, which yields
\begin{align*}
\Big|\sum_{Q\,:\, \text{dyadic}}s_{Q} \widehat{a_{Q}}(\xi)\Big| \leq C  \vert \xi\vert^{n \left( \frac{1}{\lambda}-1 \right)} \|\left\{s_{Q}\right\}_Q \|_{\lambda,q}.
\end{align*}
Since $ \widehat{f}=\sum_{Q}  s_{Q}\widehat{a_{Q}}$ converges as a distribution in  $\mathcal{S}'(\R^n)$, we conclude that $\hat{f}$ is continuous and 
\begin{equation}\nonumber
|\hat{f}(\xi)| \leq C  |\xi|^{n \left( \frac{1}{\lambda}-1 \right)} \|\left\{s_{Q}\right\}_{Q} \|_{\lambda,q}
\lesssim  |\xi|^{n \left( \frac{1}{\lambda}-1 \right)} \|f \|_{\mathcal{HM}_{q}^{\lambda}}.
\end{equation}
 
In order to conclude the proof we need to justify \eqref{convergente}.  Fixed a cube $J$ we split the sum  $\sum_{Q}=\sum_{Q\subseteq J}+\sum_{Q\supset J}$.  Using the same arguments in the proof of Proposition \ref{prop3.4}, we obtain the estimate $\Vert M_{\varphi}(a_Q)\Vert_{L^q(J)}^q\lesssim \Vert a_Q\Vert_{L^{\infty}}^q \vert Q\vert$ for $|Q| \leq |J|$. Indeed, we start writing 
\begin{align}
\Vert M_{\varphi}(a_Q)\Vert_{L^q(J)}^q =\int_{J \cap Q^{\ast}}\vert M_{\varphi}(a_Q)\vert^qdx + \int_{J\backslash Q^{\ast}}\vert M_{\varphi}(a_Q)\vert^qdx :=I_1+I_2\nonumber.
\end{align}
Mimicking the proof of Proposition \ref{prop3.4}, choosing $r>\lambda$, we estimate 
\begin{align*}
I_1 +I_2 &\lesssim \Vert a_Q\Vert^{q}_{L^{\infty}} \vert Q\vert^{q/r} \vert J \cap Q^{\ast}\vert^{1-{q}/{r}} + \| a_Q\|^{q}_{L^\infty} \ell_Q^{(n+L+1)q} \int_{2\ell_{Q}}^{\infty}r^{-q(n+L+1)+n-1}dr
\lesssim \Vert a_Q\Vert^{q}_{L^{\infty}}|Q|.
\end{align*}
Hence, we have 
\begin{align}
\vert J\vert ^{q/\lambda-1}\sum_{Q \subseteq J} \vert s_{Q}\vert^q \, \Vert M_{\varphi}(a_Q)\Vert_{L^q(J)}^q \lesssim 
\vert J\vert ^{q/\lambda-1}\sum_{Q \subseteq J} \vert s_{Q}\vert^q \, |Q|^{1-q/\lambda}\lesssim
\| \left\{  s_{Q} \right\}_Q \|_{\lambda,q}^q. \label{p1}
\end{align}
Now by  \eqref{key-est-1} and $\| a_Q\|_{L^\infty}^q\leq |Q|^{-q/\lambda}$ we can write 
\begin{align*}
|J|^{q/\lambda-1}\sum_{Q \supset J}|s_{Q}|^{q}\Vert M_{\varphi}(a_Q)\Vert_{L^q(J)}^q & \lesssim \sum_{Q \supset J}|s_{Q}|^{q} \left( |J|^{q/\lambda-1}\int_{J\cap Q^{\ast}}|M_{\varphi}(a_Q)(x)|^{q}dx\right) \\
&\lesssim  \sum_{Q \supset J}|s_{Q}|^{q} |Q|^{q/r-q/\lambda}  |J|^{q/\lambda-1}  |J \cap Q^{\ast}|^{1-{q}/{r}}\\
&\lesssim \sum_{Q \supset J} |s_{Q}|^{q} (|J|/|Q|)^{\gamma},
\end{align*}
where  $\gamma:=q\left({1}/{\lambda}-{1}/{r}\right)>0$. Fixed a dyadic cube $J$, we point out there exists a subset $N\subseteq \N$ such that  each cube $J \subset Q$ is uniquely determined by a dyadic cube $Q_{k,J} \in \big\{ Q\; \text{dyadic}: J\subset Q \text{ and }\ell_{Q}=2^{k}\ell_{J} \big\}$. Hence, we can write
$\displaystyle{ \sum_{Q\supset J} |s_{Q}|^{q} (|J|/|Q|)^{\gamma}=\sum_{k \in N} \vert s_{Q_{k,J}}\vert^q \, 2^{-kn\gamma}}$
which yields 
\begin{align}
|J|^{q/\lambda-1}\sum_{Q\supset J}|s_{Q}|^{q}\Vert M_{\varphi}(a_Q)\Vert_{L^q(J)}^q
&\lesssim \sum_{k \in N} \vert s_{Q_{k,J}}\vert^q \, 2^{-kn\gamma}\nonumber\\
&=\sum_{k \in N}\left( |s_{Q_{k,J}}|^{q}  |Q_{k,J}|^{1-q/\lambda} \right) |Q_{k,J}|^{q/\lambda-1}2^{-kn\gamma} \nonumber\\
&\leq \sum_{k \in N} \left( \sum_{Q \subseteq Q_{k,J}} |s_{Q}|^{q}  |Q|^{1-q/\lambda} \right) |Q_{k,J}|^{q/\lambda-1}2^{-kn\gamma}\nonumber \\
& \lesssim \| \left\{s_{Q} \right\}_{Q}\|_{\lambda,q}^q\sum_{k \in N}2^{-kn\gamma} \nonumber\\
& \leq C \| \left\{s_{Q} \right\}_{Q}\|_{\lambda,q}^q.\label{p2}
\end{align} 
Therefore, by \eqref{p1} and \eqref{p2}
we have
\begin{align*}
\vert J\vert ^{q/\lambda-1}\sum_{Q\,:\, \text{dyadic}} \vert s_{Q}\vert^q \Vert M_{\varphi}(a_Q)\Vert_{L^q(J)}^q \lesssim \| \left\{s_{Q} \right\}_{Q}\|_{\lambda,q}^q 
 \end{align*}
with implicit constant uniformly on $J$, as we wished to show. 
\end{proof}

\begin{remark}\label{rem1}
The Fourier transform decay \eqref{decay-Fourier2} in general is not valid  for functions on  Morrey spaces  $\mathcal{M}_q^{\lambda}(\mathbb{R}^n)$ when $1\leq q<\lambda<\infty$, in fact 
{there exists $f\in \mathcal{M}_q^{\lambda}(\mathbb{R}^n)$ such that }
	\begin{equation}\label{marceloeq}
		\vert \xi\vert^{n(1-1/\lambda)}\vert \widehat{f}(\xi)\vert \rightarrow\infty \quad \text{as}\quad  \vert\xi\vert\rightarrow\infty.
	\end{equation}
According to \cite[Remark 2.1]{deAlmeida2}, there is a special function $\varphi \in C_{c}^{\infty}(\R^n)$ such that $f(x):=\sum_{j=1}^{\infty}e^{2\pi ix\cdot x_j}\varphi(x-x_{j})$ belongs to $\mathcal{M}_q^{\lambda}(\R^n)\,$ for all $\,1\leq q<\lambda<\infty$. {However, its Fourier transform $\widehat{f}(\xi)$ satisfies \eqref{marceloeq}}. 
\end{remark}

Next we will show that constraint $ 0<q\leq \lambda\leq 1$ in Theorem \ref{cor-decay} is necessary, since the continuity of Fourier transform $\,\widehat{f}(\xi)$  breaks down at $\xi=0$ as  $\lambda>1$, for some distribution.

\begin{proposition}\label{exam.53} {There exists $f\in\mathcal{HM}_q^{\lambda}(\mathbb{R}^n)$ with $0<q\leq 1$ and $\lambda>1$ such that  $\,\widehat{f}(\xi)$  is not continuous at $\xi=0$   } 
\end{proposition}

\begin{proof} It is sufficient construct an $\mathcal{HM}_q^{\lambda}(\mathbb{R}^n)-$atom $a_Q$ such that $\widehat{a}_Q(\xi)\rightarrow\infty$ and $\xi\rightarrow0$.  According to \cite[Lemma 3.1]{HKP}, {for each $k \in \N$ there is a bounded function  $\alpha_k$ supported in the interval $[-2^{k-1}, 2^{k-1}]$ such that $\int \alpha_k(s)s^{\ell}ds=0$ for every $0 \leq \ell\leq k-1$} 
and $\|\alpha_{k}\|_{L^\infty}\leq 1$. 
Denote by $(x',x_{n})$ with $x'=(x_{1},...,x_{n-1})$ a generic point in $\R^{n} \cong \R^{n-1} \times \R$. Let us  define $a_Q(x)=2^{-k/\lambda}\psi(x')\alpha_k(x_n)$ where $\psi \in C^{\infty}_{c}(\mathbb{R}^{n-1})$ supported in the unit cube $Q'(0,1)\subset \mathbb{R}^{n-1}$ such that $\widehat{D^{\alpha}\psi}(0)=1$. Clearly $\text{supp}(a_Q)\subseteq Q:= Q'(0,1)\times [-2^{k-1}, 2^{k-1}]$, $\Vert a_Q\Vert_{\infty}\leq \vert Q\vert^{-1/\lambda}\Vert\psi\Vert_{L^{\infty}(\R^{n-1})} $ and Tonelli's theorem yields 
\begin{equation}\nonumber
\int_{\mathbb{R}^n} a_Q(x)x^{\alpha}dx = 0,\;  \text{ for all }\; \alpha \in \N^{n}_0 \; \text{such that} \;\vert\alpha\vert\leq k-1. 
\end{equation}
Taking $q\in (n/(n+k), n/(n+k-1)]$ one has $\vert\alpha\vert\leq k-1=\left\lfloor n\left({1}/{q}-1\right)\right\rfloor$ and then $a_Q$ is a {multiple of}  $\mathcal{HM}_q^{\lambda}-$atom. It follows from Proposition \ref{prop3.4} and {scaling property \eqref{scaling}} that $\|a_{\varepsilon}\|_{\mathcal{HM}_{q}^{\lambda}} = \|a_Q\|_{\mathcal{HM}_{q}^{\lambda}}\lesssim \Vert\psi\Vert_{L^{\infty}(\R^{n-1})}$, where $a_{\varepsilon}(x)=\varepsilon^{-n/\lambda} a_{Q}(x/\varepsilon)$ for every $\varepsilon>0$. According to \cite[Lemma 3.2]{HKP} we have 
\begin{equation}
\widehat{a}_{\varepsilon}(\xi',\xi_n) ={2^{-\frac{k}{\lambda}}}\varepsilon^{n(1-\frac{1}{\lambda})}\widehat{\psi}(\varepsilon\xi')\left(\frac{1-\cos(2\pi \varepsilon \xi_n)}{i\pi \varepsilon \xi_n}\right)(-i2)^k\prod_{j=1}^k\sin(2\pi j\varepsilon\xi_n)\nonumber.
\end{equation}
Then $\vert \widehat{a}_{\varepsilon}(0',1/{4k\varepsilon})\vert \simeq \varepsilon^{n(1-{1}/{\lambda})}$ which leads to $\vert \widehat{a}_{\varepsilon}(0',1/{4k\varepsilon})\vert \rightarrow \infty\;$ as $\;\varepsilon\rightarrow\infty$ if $\lambda>1$, i.e., the continuity of  $\;\widehat{a_Q}(\xi)\,$  breaks down at $\,\xi=0$. 
\end{proof}

\subsection{Sharpness on boundedness of Fourier multipliers on $\mathcal{HM}_{q}^{\lambda}$}

In this section we use the arguments in the proof of Proposition \ref{exam.53} to obtain {an optimality} of the parameters on the boundedness of Fourier multipliers on Hardy-Morrey spaces. 

Let  $\Sigma^{m}(\mathbb{R}^n)$ for $m \in \R$ the set of symbols defined by functions $\sigma\in C^{\infty}(\mathbb{R}^n\backslash\{0\})$ such that  
$$\vert \partial^{\alpha}\sigma(\xi)\vert \leq  C_{\alpha}\vert \xi\vert^{m-\vert\alpha\vert},\quad  \alpha\in\mathbb{N}_0^{n},$$
and $Op\Sigma^{m}(\mathbb{R}^n)$ the set of operators $\sigma(D)$ associated to symbols $\sigma \in \Sigma^{m}(\mathbb{R}^n)$ given by 
$$\sigma{(D)}u(x)=\int_{\mathbb{R}^{n}} e^{2\pi x \cdot \xi}\sigma(\xi)\hat{u}(\xi)d\xi, \quad \quad u \in \mathcal{S}(\mathbb{R}^{n}).$$

\begin{theorem}\cite[Theorem 2.1]{WJ}\label{jw1}
Let $\beta\geq 0$,  $0<q_{1}\leq 1$, $q_{1} \leq \lambda_{1}$ and ${\frac{1}{\lambda_{2}}:=\frac{1}{\lambda_{1}}-\frac{\beta}{n}}$. If $q_{2}\leq 1$ then 
$\sigma(D) \in Op\Sigma^{-\beta}(\mathbb{R}^n)$ is bounded from  $\mathcal{HM}_{q_1}^{\lambda_{1}}(\mathbb{R}^n)\;$ to $\;\mathcal{HM}_{q_2}^{\lambda_{2}}(\mathbb{R}^n)$ where ${q_{2}:=\lambda_{2}\frac{q_{1}}{\lambda_{1}}}$.
\end{theorem}

The next proposition is addressed  to optimality of previous result  in the sense that  if $\beta \neq n\left(\frac{1}{\lambda_{1}}-\frac{1}{\lambda_{2}}\right)$  and ${\frac{q_{2}}{\lambda_{2}}=\frac{q_{1}}{\lambda_{1}}}$, then 
there exists  $\sigma(D) \in Op\Sigma^{-\beta}(\mathbb{R}^n)$ which is not continuous from $\mathcal{HM}_{q_1}^{\lambda_{1}}(\mathbb{R}^n)\;$ to  $\;\mathcal{HM}_{q_2}^{\lambda_{2}}(\mathbb{R}^n)$.  

\begin{proposition}\label{optimal_Sigma2}  Let $0<q_{1}, q_{2} \leq 1$,  $q_{1} \leq \lambda_{1}\leq 1$ satisfying ${ \frac{q_{2}}{\lambda_{2}}=\frac{q_{1}}{\lambda_{1}}}$  and  
$\beta \neq n({1}/{\lambda_{1}}-{1}/{\lambda_{2}})$. Then there exists $\sigma(D)\in Op\Sigma^{-\beta}(\mathbb{R}^n)$ which is not continuous from $\mathcal{HM}_{q_1}^{\lambda_{1}}(\mathbb{R}^n)\;$ to $\;\mathcal{HM}_{q_2}^{\lambda_{2}}(\mathbb{R}^n)$.
\end{proposition}
\begin{proof}According the previous construction of $a_{Q}$ at Proposition  \ref{exam.53}, we consider $a_\varepsilon(x):=\varepsilon^{-n/\lambda_1}a_Q(x/\varepsilon)$ that defines 
a multiple of  an $\mathcal{HM}_{q_1}^{\lambda_1}-$atom supported on 
 $Q_{\varepsilon}:=Q'(0,\varepsilon)\times [-2^{k-1}\varepsilon, 2^{k-1}\varepsilon]$ for $k \in \N$ that satisfies $\Vert a_{\varepsilon}\Vert_{\infty}\leq \vert Q_{\varepsilon}\vert^{-1/\lambda_1}\Vert\psi\Vert_{L^{\infty}(\R^{n-1})} $ and 
\begin{equation}\nonumber
	\int_{\mathbb{R}^n} a_{\varepsilon}(x)x^{\alpha}dx = 0,\;  \text{ for all }\; \alpha \in \N^{n}_0 \; \text{ with } \;\vert\alpha\vert\leq k-1:=\left\lfloor n\left({1}/{q_1}-1\right)\right\rfloor.
\end{equation}

Let $\sigma_{\beta}(D)$ with $\sigma_{\beta}(\xi)=\vert \xi\vert^{-\beta}$ and $\beta \neq n\left(\frac{1}{\lambda_{1}}-\frac{1}{\lambda_{2}}\right)$. Now consider  $g_{\epsilon}(x) := \sigma_{\beta}(D)a_{\epsilon}(x)$ and then $
\widehat{g}_{\varepsilon}(\xi',\xi_n) =\vert\xi\vert^{-\beta}\widehat{a}_{\varepsilon}(\xi',\xi_n)$ which yields at $\vert \widehat{g}_{\varepsilon}(0',1/{4k\varepsilon})\vert \simeq\varepsilon^{n(1-{1}/{\lambda_1})+\beta}$. Hence, from Theorem  \ref{cor-decay} one has 
\begin{align}\nonumber
\Vert \sigma_{\beta}(D)a_{\varepsilon}\Vert_{\mathcal{HM}_{q_2}^{\lambda_2}} \gtrsim
\frac{\vert \widehat{g}_{\varepsilon}(0',1/{4k\varepsilon})\vert }{\left\vert (0,1/4k\varepsilon)\right\vert^{n(1/\lambda_2-1)}}\simeq  \varepsilon^{n\left(1-\frac{1}{\lambda_{1}}\right)+\beta+n\left(\frac{1}{\lambda_{2}}-1\right)}=\varepsilon^{\beta-n\left(\frac{1}{\lambda_{1}}-\frac{1}{\lambda_{2}}\right)} 
\end{align}
that blows-up, if $\beta \neq n({1}/{\lambda_{1}}-{1}/{\lambda_{2}})$. 
\end{proof}

\subsection{Hardy inequality}
{In  this section we present a Hardy-type inequality, well known for $H^p(\R^n)$ (see \cite[Corollary 7.23]{Garcia}), }extended to Hardy-Morrey spaces.

\begin{proposition} \label{coro3.7} If $f\in \mathcal{HM}_q^{\lambda}(\R^n)$ for {$0<q< \lambda\leq 1$}, then 	
	\begin{equation}\nonumber
		\left\Vert  \,{\vert \cdot \vert^{n(1-{2}/{\lambda})}} \widehat{f} \right\Vert_{ \mathcal{M}^{\lambda}_{q}}
	\lesssim \Vert f\Vert_{ \mathcal{HM}^{\lambda}_{q}},
	\end{equation}
	with implicit positive constant independent of $f$. 
\end{proposition}

\begin{proof}{The Fourier transform decay} \eqref{decay-Fourier2}  can be written as $\vert \widehat{f}(\xi)\vert\, \vert \xi \vert^{n\left(1-{2}/{\lambda}\right)} \lesssim \Vert f\Vert_{ \mathcal{HM}^{\lambda}_{q}} \vert \xi \vert^{-{n}/{\lambda}}$. Since $|\xi|^{-n/\lambda}\in \mathcal{M}_q^\lambda(\R^n)$ as $0<q<\lambda<\infty$, then 
\begin{equation}
	\left\Vert  \,{\vert \cdot \vert^{n(1-{2}/{\lambda})}}{\widehat{f}}\,\right\Vert_{ \mathcal{M}^{\lambda}_{q}} \lesssim  \left\Vert \vert \cdot\vert^{-n/\lambda}\right\Vert_{\mathcal{M}_q^{\lambda}}\Vert f\Vert_{ \mathcal{HM}^{\lambda}_{q}}=c_n\Vert f\Vert_{ \mathcal{HM}^{\lambda}_{q}}\nonumber,
\end{equation}
{as we desired.}
\end{proof}



\section{Pointwise Fourier transform decay in  $h\mathcal{M}_{q}^{\lambda}$}\label{sec4}

As consequence of Theorem \ref{cor-decay} and Theorem \ref{atomich22} we strengthen the {pointwise Fourier transform decay in $h^p(\R^n)$}  (see \cite[Proposition 5.1]{HK}) for distributions on $h\mathcal{M}_{q}^{\lambda}(\R^n)$. 
%

\begin{theorem}\label{cor-decay2}
If  $f\in h\mathcal{M}_{q}^{\lambda}(\R^n)$ for  $0<q \leq  \lambda\leq 1$ then $\widehat{f}(\xi)$ is continuous  and satisfies 
\begin{equation}\label{decay-Fourier4}
|\widehat{f}(\xi) | \lesssim  \langle\xi\rangle^{n\left(\frac{1}{\lambda}-1 \right)} \|f \|_{h\mathcal{M}_{q}^{\lambda}}, 
\end{equation}
with implicit constant independent of $f$, where $\langle\xi\rangle:=(1+\vert\xi\vert^2)^{1/2}$.
\end{theorem}
\begin{proof} By identity \eqref{XX} in the proof of Theorem \ref{atomich22} we may write
\begin{equation}\label{conv-hM}
f=\sum_{Q}s_{1,Q}a_{1,Q}+\sum_{Q}s_{2,Q}a_{2,Q} 
\end{equation}
with 
$\Vert \{s_{1,Q}\}_{Q}\Vert_{{\lambda,q}}+ \Vert \{s_{2,Q}\}_{Q}\Vert_{{\lambda,q}} \simeq \Vert \{s_Q\}_{Q}\Vert_{{\lambda,q}} \lesssim \|f\|_{h\mathcal{M}_{q}^{\lambda}}$,
where $\{a_{1,Q}\}_{Q}$ are {$\mathcal{HM}^{\lambda}_{q}-$}atoms and $\{a_{2,Q}\}_{Q}$ are  {rough $(\infty,\lambda)-$atoms with $\ell_{Q}\geq 1$}.  Moreover $\widehat{f}=\sum_Q s_{1,Q}\widehat{a_{1Q}}+\sum_Q s_{2,Q}\widehat{a_{2,Q}}\,$ is well defined in distribution sense.
 Since $f_{1}:=\sum_{Q}s_{1,Q}a_{1,Q}$ with
$\Vert f_1\Vert_{\mathcal{HM}_q^{\lambda}} \lesssim \| \left\{s_{1,Q}\right\}_{Q} \|_{\lambda,q}$
then it follows from Theorem \ref{cor-decay} that
\begin{equation}\label{t1}
\vert \widehat{f_{1}}(\xi)\vert \lesssim \vert \xi\vert ^{n\left(\frac{1}{\lambda}-1 \right)} \Vert f_1\Vert_{\mathcal{HM}_q^{\lambda}}\lesssim \vert \xi\vert ^{n\left(\frac{1}{\lambda}-1 \right)} \| \left\{s_{1,Q}\right\} \|_{\lambda,q}\lesssim \langle \xi\rangle ^{n\left(\frac{1}{\lambda}-1 \right)}  \|f\|_{h\mathcal{M}_{q}^{\lambda}},  
\end{equation}
where $|\xi|^{n\left({1}/{\lambda}-1 \right)} \leq\langle \xi\rangle ^{n\left({1}/{\lambda}-1 \right)}$ for all $\xi \in \R^n$ in view of  $\lambda\leq 1$. Now we moving on to the sum $\sum_Q s_{2,Q}\widehat{a_{2,Q}}(\xi)$, where each $a_{2,Q}$ is supported in a
cube $Q$ with center $x_{Q}$ and sidelength $\ell_{Q} \geq 1$ such that $\|a_{2,Q}\|_{L^\infty}\leq |Q|^{-1/\lambda}$. 
Consider $\varphi\in C^{\infty}_c(B(0,1))$ with nonnegative Fourier transform such that $\widehat{\varphi}(\xi)\geq 1$ for  $\xi \in  B(0,1)$ according to Lemma \ref{suitable-cutoff}. Since $\text{supp}\,(a_{2,Q})\subset Q$  and $\ell_{Q}\geq 1$ then $\text{supp} \,(\varphi_{t}\ast a_{2,Q})$ is contained (uniformly) in a cube $ \tilde{Q}$ with same center $x_{Q}$ and sidelength $c_{n}\ell_{Q}$, where $c_{n}$ is an appropriated positive constant. Follows from \eqref{control2} adapted to the cube $\widetilde{Q}$ that 
$\|\varphi_{t}\ast a_{2,Q}\|_{L^\infty}\lesssim \ell_{Q}^{-n/\lambda}\|a_{2,Q}\|_{h\mathcal{M}^{\lambda}_{q}}$. Hence, 
\begin{equation}\label{a2parte1}
\|\varphi_{t} \ast a_{2,Q}\|_{L^{1}} \leq |\widetilde{Q}| \, \|\varphi_{t}\ast a_{2,Q}\|_{L^\infty} \lesssim \ell_{Q}^{n(1-1/\lambda)} \|a_{2,Q}\|_{h\mathcal{M}^{\lambda}_{q}}, 
\end{equation}
uniformly for $0<t\leq 1$. {Now taking $t=|\xi|^{-1}$ which implies $|\xi|\geq \ell_Q^{-1}$ for $\ell_Q\geq 1$ and $0<t\leq 1$, we  can proceed as  \eqref{key-est-for -hM} and invoking \eqref{a2parte1} we have} 
\begin{align*}
\vert \widehat{a_{2,Q}}(\xi)\vert\leq  \big\vert\widehat{\varphi}(\xi/|\xi|)\widehat{a_{2,Q}}(\xi)\big\vert 
= \big \vert ({\varphi_{{|\xi|^{-1}}} \ast a_{2,Q}})^{\wedge}(\xi) \vert
& \leq  \|{\varphi_{{|\xi|^{-1}}} \ast a_{2,Q}}\|_{L^{1}}  \\
&\lesssim \ell_{Q}^{n(1-1/\lambda)} \|a_{2,Q}\|_{h\mathcal{M}^{\lambda}_{q}} \\  
&\lesssim \Vert a_{2,Q}\Vert_{h\mathcal{M}_{q}^{\lambda}} \\
& \lesssim \langle \xi\rangle ^{n\left({1}/{\lambda}-1 \right)}   \Vert a_{2,Q}\Vert_{h\mathcal{M}_{q}^{\lambda}},
\end{align*}
since $\ell_{Q} \geq 1$ and $0<\lambda \leq 1$.  In order to control $\,\widehat{a_{2,Q}}(\xi)$ in the set $Q^{\sharp}:= \{ \xi \in \R^n: |\xi|<\ell_{Q}^{-1} \} \subseteq B(0,1)$ we take \eqref{a2parte1} to $t=1$ and then  
\begin{align*}
\vert \widehat{a_{2,Q}}(\xi)\vert 
\lesssim \big\vert\widehat{\varphi}(\xi)\widehat{a_{2,Q}}(\xi)\big\vert 
= \big \vert \widehat{\varphi \ast a_{2,Q}}(\xi)\big\vert  
\leq \|\varphi \ast a_{2,Q}\|_{L^{1}}
&\lesssim  \ell_{Q}^{n(1- 1/\lambda)} \Vert a_{2,Q}\Vert_{h\mathcal{M}^{\lambda}_{q}} \\
&\lesssim  \langle \xi\rangle ^{n\left({1}/{\lambda}-1 \right)} \Vert a_{2,Q}\Vert_{h\mathcal{M}^{\lambda}_{q}}.
\end{align*}
{Therefore, mimicking the proof of  {Theorem \ref{cor-decay}}, we may write $f_{2}:=\sum_{Q}s_{1,Q}a_{1,Q}$ where
$\Vert f_2\Vert_{h\mathcal{M}_q^{\lambda}} \lesssim \| \left\{s_{2,Q}\right\} \|_{\lambda,q}$   such that
 \begin{equation}\label{t2}
\vert \widehat{f_{2}}(\xi)\vert \lesssim \langle \xi\rangle ^{n\left(\frac{1}{\lambda}-1 \right)} \| \left\{s_{2,Q}\right\} \|_{\lambda,q} \lesssim  \langle \xi\rangle ^{n\left(\frac{1}{\lambda}-1 \right)}   \|f\|_{h\mathcal{M}_{q}^{\lambda}}.
\end{equation}
The estimate \eqref{decay-Fourier4} follows from \eqref{t1} and \eqref{t2}.}
 \end{proof}

\begin{remark} If $a_{Q}$ is {an} $h\mathcal{M}_q^{\lambda}-$atom associated to the cube $Q(x_{Q},r_{Q})$ for $0<q < \lambda\leq 1$, we may conclude {from the proof of Theorem \ref{cor-decay2}} the necessity of cancellation type condition 
$$\left\vert \int_{\R^{n}}a_{Q}(x)(x-x_{Q})^{\alpha}dx \right\vert \lesssim  1, \quad |\alpha|\  \leq   \left\lfloor n\left({1}/{\lambda}-1\right)\right\rfloor. $$
Controls of this type on localizable Hardy spaces have been investigated by Dafni et all \cite{Dafni, Dafni2, Dafni4, Dafni3}, in which the vanish moment conditions of Goldberg`s atoms associated to balls were relaxed to an approximated condition. 
\end{remark}

\subsection{Sharpness on boundedness in  the class $OpS^{m}_{1,0}(\R^n)$ on $h\mathcal{M}_q^{\lambda}$} 

For $m \in \mathbb{R}$, recall that a symbol $a=a(x,\xi) \in S^{m}_{1,0}(\mathbb{R}^n)$ 
of order $m$ and type $(1,0)$ is a smooth function defined on $\mathbb{R}^n\times \mathbb{R}^n$ satisfying the following estimate
\begin{equation}
	\vert\partial^{\alpha}_{x}\partial^{\beta}_{\xi}a(x,\xi)\vert\leq C_{\alpha, \beta}\left\langle \xi\right\rangle^{m-\vert\beta\vert},\;\; \alpha,\beta \in \Z_{+}^{n}. \nonumber
\end{equation}
To each symbol $a(x,\xi) \in S^{m}_{1,0}(\R^{n})$ we associate the pseudodifferential operator $a(x,D) \in OpS^{m}_{1,0}(\R^{n})$ given  by 
\begin{equation}\nonumber
a(x,D)u(x)=\int e^{2\pi i x \cdot \xi} a(x,\xi) \hat{u}(\xi)d\xi,\;\;\;\;u \in \mathcal{S}'(\R^{n}).
\end{equation}

\begin{theorem}[\cite{Dachun3}]\label{thm-Op}
	Let  $0<q\leq 1$ and $q< \lambda<\infty$. Then the operator $b(x,D) \in OpS^{0}_{1,0}(\mathbb{R}^{n})$ maps continuously $h\mathcal{M}^{\lambda}_{q}(\mathbb{R}^{n})$ into itself.
\end{theorem}

As an application of  {Theorem \ref{cor-decay2}}, we may conclude that pseudodifferential operators in the H\"ormander class $OpS^{m}_{1,0}(\R^n)$ can not be bounded in $h\mathcal{M}_q^{\lambda}(\R^n)$ for $0< q<\lambda\leq 1$ and $m>0$. Indeed, consider the pseudodifferential operator $J_m(D)$ associated to the symbol $\langle \xi\rangle^{m}\in S_{1,0}^m(\R^n)$ with $m>0$.
 {According to the proof of Proposition  \ref{exam.53}, there exists $a_\varepsilon(x):=\varepsilon^{-n/\lambda}a_Q(x/\varepsilon)$ an $h\mathcal{M}_q^{\lambda}-$atom with $\|a_{\epsilon}\|_{h\mathcal{M}_q^{\lambda}} \simeq \|a_{\epsilon}\|_{\mathcal{HM}_q^{\lambda}}\lesssim \|\psi\|_{L^{\infty}}\lesssim 1$ for $0<\epsilon \leq 1 $.
Let  $g_{\epsilon}(x) := J_{m}(D)a_{\epsilon}(x)$ that satisfies $
		\widehat{g}_{\varepsilon}(\xi',\xi_n) =\langle\xi\rangle^{m}\widehat{a}_{\varepsilon}(\xi',\xi_n)$, then $\vert \widehat{g}_{\varepsilon}(0,1/{4k\varepsilon})\vert \simeq \varepsilon^{n(1-{1}/{\lambda})-m}$. Hence, from Theorem  \ref{cor-decay2} one has 
		\begin{equation}\nonumber
			\Vert J_m a_{\varepsilon}\Vert_{h\mathcal{M}_q^{\lambda}} \gtrsim  \frac{\vert \widehat{g}_{\varepsilon}(0,1/{4k\varepsilon})\vert }{\langle(0,1/4k\varepsilon)\rangle^{n(1/\lambda-1)}} \simeq \varepsilon^{-m} \rightarrow \infty\; ,
		\end{equation}
		as $\varepsilon \rightarrow 0^{+}$, thus the Theorem \ref{thm-Op}  can not be extended for operators in the class $OpS^{m}_{1,0}(\mathbb{R}^{n})$ for $0<q <\lambda\leq 1$ with $m>0$}.  Moreover, in the same spirit of Proposition \ref{optimal_Sigma2}, the previous argument can be extended to show that if $0<q_{1}, q_{2} \leq 1$,  $q_{1} \leq \lambda_{1}\leq 1$ satisfies $\frac{q_{2}}{\lambda_{2}}=\frac{q_{1}}{\lambda_{1}}$  and  
$0<m < n({1}/{\lambda_{1}}-{1}/{\lambda_{2}})$ then there exists $\sigma(x,D)\in OpS^{-m}_{1,0}(\mathbb{R}^n)$ which is not continuous from $h\mathcal{M}_{q_1}^{\lambda_{1}}(\mathbb{R}^n)\;$ to $\;h\mathcal{M}_{q_2}^{\lambda_{2}}(\mathbb{R}^n)$. Hence, the condition $m\geq  n({1}/{\lambda_{1}}-{1}/{\lambda_{2}})$ is natural. The case $m=0$ was settled by  \cite{Dachun3}.

\noindent \textbf{Acknowledgment.} We would like to thank Prof. Yoshihiro Sawano (Department of Mathematics and Information Sciences, Tokyo Metropolitan University) for carefully reading through the drafts of this paper.

\end{document}